\documentclass[10pt]{amsart}
\usepackage{amsmath, amsthm, amssymb, amsfonts, verbatim}
\usepackage{array}
\usepackage[bookmarks]{hyperref}
\usepackage[active]{srcltx}
\usepackage{amscd}
\usepackage{epsfig}
\usepackage{color}
\usepackage{comment}
\usepackage{setspace}
\usepackage{textcomp}
\usepackage{stmaryrd}
\usepackage{undertilde}

\numberwithin{equation}{section}

\newcommand{\ra}{\rightarrow}
\newcommand{\lb}{\llbracket}
\newcommand{\rb}{\rrbracket}

\newcommand{\half}{\frac{1}{2}}

\newcommand{\e}{\epsilon}

\newcommand{\R}{{\mathbb R}}

\newcommand{\X}{{\mathbb X}}

\newcommand{\calT}{\mathcal{T}}
\newcommand{\calP}{\mathcal{P}}

\renewcommand{\div}{\operatorname{div}}

\newcommand{\sym}{\operatorname{sym}}

\newcommand{\tr}{\operatorname{tr}}

\newcommand{\esssup}{\operatorname{esssup}}

\newcommand{\tv}{\theta_{\bs{z}}}
\newcommand{\tp}{\theta_p}

\newcommand{\bs}{\boldsymbol}

\newcommand{\lap}{\Delta}
\newcommand{\pd}{\partial}

\newtheorem{thm}{Theorem}[section]
 
\newtheorem{lemma}[thm]{Lemma}

\newtheorem{defn}[thm]{Definition}
\newtheorem{rmk}[thm]{Remark}

\begin{document}

\title[Finite element method of Biot's consolidation model]{Guaranteed locking-free finite element methods for Biot's consolidation model in poroelasticity}
\author{Jeonghun J. Lee}
\maketitle

\begin{abstract}
We propose a new finite element method for the three-field formulation of Biot's consolidation model in poroelasticity and prove a priori error estimates. Uniform-in-time error estimates of all the unknowns are obtained for both semidiscrete solutions and fully discrete solutions with the backward Euler time discretization. 
The novelty of our method is that the error analysis does not require the assumption that the constrained specific storage coefficient is uniformly positive. Therefore the method is guaranteed to be locking-free without the additional assumption on material parameters. 
\end{abstract}

\section{Introduction}
The Biot's consolidation model simultaneously describes the deformation of a saturated elastic porous medium and the viscous fluid flow inside \cite{MR0066874}. Since this model naturally arises from classical engineering applications such as geomechanics and petrolium engineering, the study of its numerical solutions has been of great interest. 

There are extensive literature on numerical schemes for Biot's consolidation model with finite element methods. In early studies of the problem with continuous Galerkin finite elements, nonphysical pressure oscillations of numerical solutions, called poroelasticity locking, were observed for certain ranges of material parameters and small time step sizes \cite{NAG:NAG1610080304,MR603175,NAG:NAG1610080106}.

In order to avoid the poroelasticity locking, various numerical methods for the problem with different formulations were considered. In a series of papers \cite{MR1156589,MR1257948,MR1393902}, Murad and his collaborators studied a two-field formulation of Biot's model in incompressible porous media, with displacement and pressure as unknowns, using mixed finite elements for Stokes equations. A discontinuous Galerkin method for the two-field formulation was studied recently \cite{MR3047799}. A Galerkin least square method \cite{MR2177147} was proposed for a four-field formulation of the problem, which has displacement, stress, fluid flux, and pressure as unknowns. A three-field formulation was studied with various couplings of continuous and discontinuous Galerkin methods, and mixed finite element methods in \cite{MR2327964,MR2327966,MR2461315}.  A coupling of nonconforming and mixed finite element methods for the formulation was recently studied in \cite{NUM:NUM21775}. For more information on previous studies we refer to \cite{MR3047799,lewis1998finite,NUM:NUM21775} and the references therein.

In the present paper, we are interested in the three-field formulation of the problem. Compared to the two-field formulation approach, the one more unknown seems to be a disadvantage but there is a reason that the three-field formulation approach may be preferred. An exact solution of Biot's model typically has a large jump of pressure, so when continuous finite elements are used for the numerical pressure, there is an overshoot for this large jump. Unfortunately, both the Taylor--Hood and MINI elements, the most popular mixed finite element families for Stokes problems, use continuous finite elements for pressure, so the overshooting is inevitable. However,  in the methods for three-field formulation, discontinuous finite elements are eligible for the pressure, so there is no overshooting. 

To ensure that a method is locking-free, a uniform-in-time error bound of the pressure is needed. In addition, the heuristic analysis and numerical experiments in \cite{phillips-wheeler} insinuate that the locking typically occurs when the constrained specific storage coefficient $c_0 \geq 0$ of the model is very close to 0. Thus the uniform-in-time error bound should be robust for very small or even vanishing $c_0$. However, the assumption that $c_0$ is uniformly positive is essential to obtain a uniform-in-time error bound of pressure in the error analysis of aforementioned methods for the three-field formulation \cite{MR2327964,MR2327966,NUM:NUM21775}. Although numerical tests of those methods suggest that they are good candidates to avoid the poroelasticity locking, their error analysis is not satisfactory from the mathematical point of view.

The goal of this paper is to develop a new locking-free finite element method for the three-field formulation of Biot's model. To confirm that the method is locking-free we show a priori error estimates of uniform-in-time errors of all the unknowns without the assumption that $c_0$ is uniformly positive. To the best of our knowledge, this result has not been achieved previously.

The paper is organized as follows. In section \ref{sec:prelim} we introduce notations and preliminaries of the problem. In section \ref{sec:semidiscrete-analysis} we present our numerical method for spatial discretization of the problem and show a priori error analysis of semidiscrete solutions. In section \ref{sec:fullydiscrete-analysis} we show error analysis of fully discrete solutions with the backward Euler time discretization. Finally, we conclude with remarks on extending the method to higher order methods and to rectangular meshes. 

\section{Preliminaries} \label{sec:prelim}

\subsection{Notations}
Let $\Omega$ be a bounded Lipschitz domain in $\R^n$ with $n=2$ or $3$. 
For a nonnegative integer $m$, $H^m (\Omega)$, $H^m(\Omega; \R^n)$ denote the standard $\R$ and $\R^n$-valued Sobolev spaces based on the $L^2$ norm ($H^0 = L^2$). For a set $G \subset \Omega$, $\| \cdot \|_{m,G}$ is the $H^m$ norm on $G$. If $G = \Omega$, then we simply use $\| \cdot \|_m$. The set of functions in $L^2(\Omega; \R^n)$ whose divergence are in $L^2(\Omega)$, and the corresponding norm are denoted by $H(\div, \Omega)$ and $\| \cdot \|_{\div}$. 

For a reflexive Banach space $\mathcal{X}$ and $0<T_0<\infty$, $C^0 ([0,T_0]; \mathcal{X})$ denotes the set of functions $f : [0,T_0] \rightarrow \mathcal{X}$ which are continuous in $t \in [0,T_0]$. For an integer $m \geq 1$ we define 
\begin{align*}
C^m ([0,T_0]; \mathcal{X}) = \{ f \, | \, \pd^{i}f/\pd t^{i} \in C^0([0,T_0];\mathcal{X}), \, 0 \leq i \leq m \},
\end{align*}
where $\pd^i f/\pd t^i$ is the $i$-th time derivative in the sense of the Fr\'echet derivative in $\mathcal{X}$ (see e.g., \cite{Yosida-book}). For a function $f : [a,b] \ra \mathcal{X}$, we define the space-time norm 
\begin{align*}
\| f \|_{L^r([a,b]; \mathcal{X})} = 
\begin{cases}
\left( \int_a^b \| f \|_\mathcal{X}^r ds \right)^{1/r}, \quad 1 \leq r < \infty, \\
\esssup_{t \in [a,b]} \| f \|_\mathcal{X}, \quad r = \infty.
\end{cases}
\end{align*}
If the time interval is fixed as $[0,T_0]$, then we use $L^r \mathcal{X}$ to denote $L^r([0,T_0]; \mathcal{X})$ for simplicity. We define the space-time Sobolev spaces $W^{k,r}([0,T_0]; \mathcal{X})$ for nonnegative integer $k$ and $1 \leq r \leq \infty$ as the closure of $C^k ([0,T_0]; \mathcal{X})$ with the norm $\| f \|_{W^{k,r} \mathcal{X}} = \sum_{i=0}^k \| \pd^i f / \pd t^i \|_{L^r \mathcal{X}}$. The Sobolev embedding \cite{MR697382} gives
\begin{align} \label{eq:sobolev}
W^{k+1,1} \mathcal{X} \hookrightarrow W^{k,\infty} \mathcal{X}. 
\end{align}
We adopt a convention that $\| f, g \|_\mathcal{X} = \| f \|_\mathcal{X} + \| g \|_\mathcal{X}$ for the norm of a Banach space $\mathcal{X}$. For two Banach spaces $\mathcal{X}$ and $\mathcal{Y}$, and for $f \in \mathcal{X} \cap \mathcal{Y}$, $\| f \|_{\mathcal{X} \cap \mathcal{Y}}$ will stand for $\| f \|_{\mathcal{X}} + \| f \|_{\mathcal{Y}}$. For simplicity of notations, $\dot{f}$, $\ddot{f}$, $\dddot{f}$ will be used to denote time derivatives $\pd f/ \pd t$, $\pd^2 f / \pd t^2$, $\pd^3 f / \pd t^3$. 


A shape-regular triangulation of $\Omega$ will be denoted by $\calT_h$ for which $h$ is the maximum diameter of triangles (or tetrahedra) and $\mathcal{E}_h$ is the corresponding set of edges (faces), respectively. The interior edges/faces $\mathcal{E}_h^{\circ}$ is the set $\{ E \in \mathcal{E}_h \,|\, E \subset \Omega \}$. For $E \in \mathcal{E}_h$ and functions $\bs{f}, \bs{g} : \mathcal{E}_h \ra \R^n$ we define
\begin{align*}
\langle \bs{f}, \bs{g} \rangle_E = \int_E \bs{f} \cdot \bs{g} \,ds, \qquad \langle \bs{f}, \bs{g} \rangle = \sum_{E \in \mathcal{E}_h} \langle \bs{f}, \bs{g} \rangle_E. 
\end{align*}
For $E \in \mathcal{E}_h$ and an element-wise $H^1$ function $\bs{v}$, $\lb \bs{v} \rb$ is defined by 
\begin{align*}
\lb \bs{v} \rb|_E = 
\begin{cases}
\text{the jump of }\bs{v} \text{ on } E, \quad &\text{ if } E \in \mathcal{E}_h^{\circ}, \\
\bs{v}, \quad &\text{ if }E \subset \pd \Omega.
\end{cases}
\end{align*}
For an integer $k \geq 0$, and $G \subset \R^n$, $\mathcal{P}_k(G)$ is the space of polynomials defined on $G$ of degree $\leq k$. We use $\calP_k (\calT_h)$ to denote the space of piecewise polynomials on $\mathcal{T}_h$ of degree $\leq k$. For a vector space $\X$, we use $\mathcal{P}_k(G; \X)$ and $\mathcal{P}_k(\mathcal{T}_h; \X)$ to denote the space of $\X$-valued polynomials with same conditions. 

Throughout this paper we use $X \lesssim Y$ to denote the inequality $X \leq cY$ with a generic constant $c>0$ which is independent of the mesh size, and $X \sim Y$ will stand for $X \lesssim Y$ and $Y \lesssim X$. If needed, we will use $c$ to denote generic positive constants in inequalities and it can be a different constant in every line.

\subsection{The Biot's consolidation model}
In this subsection we review the Biot's consolidation model in poroelasticity. In an elastic porous medium saturated with a fluid, fluid flow and deformation of the porous medium are intimately related and their simultaneous behaviors are described by Biot's model.

Throughout this paper we restrict our interest to quasistatic consolidation problems. In other words, we assume that the consolidation process is slow and the acceleration term is ignored. In our description of the model, $\bs{u}$ is the displacement of the porous medium, $p$ is the fluid pressure, $\bs{f}$ is the body force, and $g$ is the source/sink density function of the fluid. The governing equations of the model are 
\begin{align}
\label{eq:strong-eq1}-\div \mathcal{C} \e(\bs{u}) + \alpha \nabla p &= \bs{f}, \\
\label{eq:strong-eq2} c_0 \dot{p} + \alpha \div \dot{\bs{u}} - \div (\utilde{\bs{\kappa}} \nabla p) &= g, 
\end{align}
where $\mathcal{C}$ is the elastic stiffness tensor, $c_0 \geq 0 $ is the constrained specific storage coefficient, $\utilde{\bs{\kappa}}$ is the hydraulic conductivity tensor, and $\alpha>0$ is the Biot--Willis constant which is close to 1. In \eqref{eq:strong-eq1}, the $\div$ is the row-wise divergence of the $\R^{n\times n}$-valued function $\mathcal{C}\e(\bs{u})$. 

For isotropic elastic porous media, the elasticity tensor $\mathcal{C}$ has the form
\begin{align*} 
\mathcal{C} \utilde{\bs{\tau}} = 2 \mu \utilde{\bs{\tau}} + \lambda \tr(\utilde{\bs{\tau}}) \utilde{\bs{I}}, \qquad \utilde{\bs{\tau}} \in L^2(\Omega; \R_{\sym}^{n \times n}),
\end{align*}
where the constants $\mu, \lambda >0$ are Lam\'{e} coefficients, $\utilde{\bs{I}}$ is the identity matrix, and $\R_{\sym}^{n \times n}$ is the space of symmetric $n \times n$ matrices. We assume that $\mu$, $\lambda$ are bounded from above and below. 
The coefficient $c_0 \geq 0 $ is determined by the permeability of the porous medium, and the bulk moduli of the solid and the fluid. The hydraulic conductivity tensor $\utilde{\bs{\kappa}}$ is defined by the permeability tensor of the solid divided by the fluid viscosity and it is positive definite. We assume that $\utilde{\bs{\kappa}}$ is uniformly bounded from above and below. For derivation of these equations from physical modeling, we refer to standard porous media references, for instance, \cite{anandarajah2010computational}.

In order to be a well-posed problem, the equations (\ref{eq:strong-eq1}--\ref{eq:strong-eq2}) need appropriate boundary and initial conditions. We assume that there are two partitions of $\pd \Omega$,
\begin{align*}
\pd \Omega = \Gamma_p \cup \Gamma_f, \qquad \pd \Omega = \Gamma_d \cup \Gamma_t,
\end{align*}
with $| \Gamma_p |,| \Gamma_d | > 0$, i.e., the $n-1$-dimensional measure of $\Gamma_p$ and $\Gamma_d$ are positive. Boundary conditions are given by
\begin{align}
\label{eq:bc}  
\begin{split}
&p(t) = 0 \quad \text{ on } \Gamma_p, \quad - \utilde{\bs{\kappa}} \nabla p(t) \cdot \bs{n} = 0 \quad\text{ on } \Gamma_f, \\
& \bs{u}(t) = 0 \quad\text{ on } \Gamma_d, \quad \utilde{\bs{\sigma}}(t) \bs{n} = 0 \quad\text{ on } \Gamma_t,
\end{split}
\end{align}
for all $t \in [0, T_0]$, in which $\bs{n}$ is the outward unit normal vector field on $\pd \Omega$ and $\utilde{\bs{\sigma}}(t) := \mathcal{C} \e(\bs{u}(t)) - \alpha p(t) \utilde{\bs{I}}$.
Here we only consider this homogeneous boundary condition for simplicity but our method, which will be introduced later, can be extended readily to problems with inhomogeneous boundary conditions. We also assume that given initial data $p(0), \bs{u}(0)$ and initial body force $\bs{f}(0)$ satisfy \eqref{eq:strong-eq1}.

Regularity of the solutions of the problem (\ref{eq:strong-eq1}--\ref{eq:strong-eq2}) with suitable boundary conditions were thoroughly studied in \cite{MR1790411}. 
When we claim a priori error estimates we assume that exact solutions are sufficiently regular to obtain the claimed error bounds. For simplicity, we assume that $\alpha=1$ and $\utilde{\bs{\kappa}} = \utilde{\bs{I}}$ in the rest of the paper. However, we assume that $c_0 \geq 0$ is only bounded from above.


\subsection{Variational formulation}
Assuming $\alpha = 1$, $\utilde{\bs{\kappa}} = \utilde{\bs{I}}$, and introducing a new unknown $\bs{z} := \utilde{\bs{\kappa}} \nabla p$ in (\ref{eq:strong-eq1}--\ref{eq:strong-eq2}), we have 
\begin{align}
\label{eq:new-strong-eq1} -\div \mathcal{C} \e(\bs{u}) + \nabla p &= \bs{f}, \\
\label{eq:new-strong-eq2} \bs{z} - \nabla p &= 0, \\
\label{eq:new-strong-eq3} c_0 \dot{p} + \div \dot{\bs{u}} - \div \bs{z} &= g.
\end{align}
Let 
\begin{align*}
\Sigma_{\Gamma_d} &= \{ \bs{u} \in H^1(\Omega; \R^n) \,|\, \bs{u}|_{\Gamma_d} = 0 \}, \\
V_{\Gamma_f} &= \{ \bs{z} \in H(\div, \Omega) \,|\, \bs{z}\cdot \bs{n}|_{\Gamma_f} = 0 \}, \\
W &= L^2(\Omega).
\end{align*}
and define a bilinear form 
\begin{align*}
a(\bs{u}, \bs{v}) = (\mathcal{C} \e(\bs{u}), \e(\bs{v})) = 2\mu (\e({\bs{u}}), \e({\bs{v}})) + \lambda (\div \bs{u}, \div \bs{v}), \quad \bs{u}, \bs{v} \in H^1(\Omega; \R^n).
\end{align*}
Then a variational formulation of (\ref{eq:new-strong-eq1}--\ref{eq:new-strong-eq3}) with boundary conditions \eqref{eq:bc} is to seek $(\bs{u}, p) \in C^1([0,T_0]; \Sigma_{\Gamma_d} \times W)$ and $\bs{z} \in C^0([0,T_0];V_{\Gamma_f})$ such that 
\begin{align}
\label{eq:weak-eq1} a(\bs{u}, \bs{v}) - (p, \div \bs{v}) &= (\bs{f}, \bs{v}), & & \bs{v} \in \Sigma_{\Gamma_d}, \\
\label{eq:weak-eq2} ( \bs{z}, \bs{w}) + (p, \div \bs{w}) &= 0, & & \bs{w} \in V_{\Gamma_f}, \\
\label{eq:weak-eq3} (c_0 \dot{p}, q) +(\div \dot{\bs{u}}, q) - (\div \bs{z}, q) &= (g, q), & & q \in W.
\end{align}

\subsection{Finite element spaces}
For discretization of (\ref{eq:weak-eq1}--\ref{eq:weak-eq3}) we need three finite element spaces $\Sigma_h$, $V_h$, $W_h$ for unknowns $\bs{u}$, $\bs{z}$, $p$, respectively.

We define $V_h$ and $W_h$ as the lowest order Raviart--Thomas--Ned\'{e}l\'{e}c space and piecewise constant finite element space
\begin{align*}
V_h &= \{ \bs{w} \in V_{\Gamma_f} \,|\, \bs{w}|_T \in (\mathcal{P}_0(T)^n + \bs{x} \,\mathcal{P}_0(T)), \;\forall T \in \mathcal{T}_h \}, \\
W_h &= \{ q \in L^2(\Omega) \,|\, q|_T \in \mathcal{P}_0(T), \;\forall T \in \mathcal{T}_h \}.
\end{align*}
Here $\bs{x}$ is the vector field $(x_1\;x_2)^T$ and $(x_1\;x_2\;x_3)^T$ in two and three dimensions, respectively. 
Let $\Pi_h^{RT}$ be the canonical Raviart--Thomas interpolation operator into $V_h$, and $Q_h$ be the orthogonal $L^2$ projection into $W_h$. It is well-known that, for $\bs{w} \in H^1(\Omega; \R^n)$ and $q \in H^1(\Omega)$,
\begin{gather}
\label{eq:commute1} \div V_h = W_h, \qquad \div \Pi_h^{RT} \bs{w} = Q_h \div \bs{w},   \\
\label{eq:approx1} \| \bs{w} - \Pi_h^{RT} \bs{w} \|_0 \lesssim h \| \bs{w} \|_{1}, \qquad \| q - Q_h q \|_0 \lesssim h \| q \|_{1},
\end{gather}
hold. Furthermore, there exists a $\beta >0$, independent of the mesh size, such that 
\begin{align}
\label{eq:inf-sup1} \inf_{0 \not = q \in W_h} \sup_{0 \not = \bs{w} \in V_h} \frac{(q, \div \bs{w})}{\| q \|_0 \| \bs{w} \|_{\div}} \geq \beta > 0.
\end{align}

For $\Sigma_h$ we use vector-valued nonconforming $H^1$ elements. In two dimensions we use 
the Mardal--Tai--Winther element \cite{MR1950614} which has local shape functions on $T \in \mathcal{T}_h$ as 
\begin{align*}
\Sigma_T = \{ \bs{v} \in (\mathcal{P}_3(T))^2 \,|\, \div \bs{v} \in \mathcal{P}_0(T), \;\bs{v} \cdot \bs{n}|_E \in \mathcal{P}_1(E), \;E \subset \pd T \},
\end{align*}
and DOFs 
\begin{align*}
\bs{v} \mapsto \int_{E} \bs{v} \cdot \bs{n}_E r \,ds, \quad \forall r \in \mathcal{P}_1(E), \qquad 
\bs{v} \mapsto \int_E \bs{v} \cdot \bs{t}_E \,ds,
\end{align*}
in which $\bs{n}_E$ and $\bs{t}_E$ are the unit normal and tangential vectors on an edge $E$ of $T$. Let $\Sigma_h$ be the Mardal--Tai--Winther element with an additional condition that  
\begin{align} \label{eq:sigma-property5}
\text{all DOFs associated to the edges in } \Gamma_d \text{ vanish}.
\end{align}
By definition one can see that $\Sigma_h \subset H(\div, \Omega)$ and $\div \Sigma_h \subset W_h$. In three dimensions we use the element developed in \cite{MR2283095} which is a three dimensional analogue of the Mardal--Tai--Winther element.

The rest of this section will be devoted to present properties of $\Sigma_h$ which will be important for well-posedness and the a priori error analysis of our methods. Let us first define a discrete semi-norm for element-wise $H^1$ functions by
\begin{align*}
\| \bs{v} \|_{1,h}^2 = \sum_{T \in \mathcal{T}_h} \| \nabla \bs{v} \|_{0,T}^2. 
\end{align*}
By the Poincar\'{e} inequality this is in fact a norm on $\Sigma_{\Gamma_d}$. By \eqref{eq:sigma-property5}, the interelement continuity of $\Sigma_h$, and a discrete Poincar\'{e} inequality \cite{MR1974504}, $\| \cdot \|_{1,h}$ is a norm on $\Sigma_h$ as well. 
Let 
\begin{align*}
\Sigma_{\Gamma_d} + \Sigma_h = \{ \bs{v} \,|\, \bs{v} = \bs{v}_1 + \bs{v}_2 \quad \text{for } \bs{v}_1 \in \Sigma_{\Gamma_d}, \bs{v}_2 \in \Sigma_h \}.
\end{align*}
Then we are able to prove the following discrete Korn's inequality.
\begin{lemma} \label{lemma:korn} For $\bs{v} \in \Sigma_{\Gamma_d} + \Sigma_h$ and the element-wise symmetric gradient $\e_h$,
\begin{align*}
\| \bs{v} \|_{1,h} \sim \| \e_h(\bs{v}) \|_0 .
\end{align*}
\end{lemma}
\begin{proof}
By the definition of $\| \cdot \|_{1,h}$, $\| \e_h(\bs{v}) \|_{0} \leq \| \bs{v} \|_{1,h}$ is obvious. 

Proof of the other direction is essentially same to the proof of Theorem 3.1 in \cite{MR2176387}, so we only sketch it. The inequality (1.12) in \cite{Brenner04} gives 
\begin{align*}
\| \bs{v} \|_{1,h}^2 &\lesssim \| \e_h(\bs{v}) \|_0^2 + \sum_{E \in \mathcal{E}_h, E \subset \Omega} h_E^{-1} \| \Pi_E \lb \bs{v} \rb \|_{0,E}^2 \\
&\quad + \sup_{\stackrel{\bs{m} \in RM(\Omega)} {\| \bs{m} \|_{0,\Gamma_d} = 1, \int_{\Gamma_d} \bs{m} ds = 0}} \left( \int_{\Gamma_d} \bs{v} \cdot \bs{m} ds \right)^2,
\end{align*}
where $\Pi_E$ is the $L^2$ projection to $\mathcal{P}_1(E; \R^2)$ and $RM(\Omega)$ is the space of rigid body motions on $\Omega$. Actually, the last term in the above vanishes because $\bs{v} \in  \Sigma_{\Gamma_d} + \Sigma_h$. In the second term $\Pi_E \lb \bs{v} \rb$ can be reduced to the jumps projected to the space of traces of rigid body motions on $E$ under the help of $\| \e_h(\bs{v}) \|_0^2$ as in \cite{MR2176387}. However, there is no jumps in traces of rigid body motions due to the interelement continuity of $\Sigma_h$, so the desired inequality follows. 
\end{proof}
It is proven in \cite{MR1950614} that there exists an interpolation $\Pi_h : H^2(\Omega; \R^n) \ra \Sigma_h$, satisfying 
\begin{align}
\label{eq:sigma-property1} &  \div \Pi_h \bs{v} = Q_h \div \bs{v}, \\
\label{eq:sigma-property2}& \| \bs{v} - \Pi_h \bs{v} \|_0 \lesssim h \| \bs{v} \|_1, \quad \| \Pi_h \bs{v} \|_{1,h} \lesssim \| \bs{v} \|_1, \quad \| \Pi_h \bs{v} - \bs{v} \|_{1,h} \lesssim h\| \bs{v} \|_2. 
\end{align}
Finally, there exists a $\beta'>0$ such that
\begin{align}
\label{eq:sigma-property3} \inf_{0 \not = q \in W_h} \sup_{0 \not = \bs{v} \in \Sigma_h} \frac{(q, \div \bs{v})}{\| q \|_0 \| \bs{v} \|_{1,h}} \geq \beta' > 0. 
\end{align}


%
%
\section{Error analysis of semidiscrete solutions} \label{sec:semidiscrete-analysis}
In this section we discuss well-posedness of the semidiscrete problems of (\ref{eq:new-strong-eq1}--\ref{eq:new-strong-eq3}) and the a priori error analysis of the semidiscrete solutions. In the rest of this paper $\Sigma_h \times V_h \times W_h$ will always be the finite element spaces introduced in the previous section. 
\subsection{Semidiscrete problem and its well-posedness}
The semidiscrete problem of (\ref{eq:weak-eq1}--\ref{eq:weak-eq3}) is to seek $(\bs{u}_h, \bs{z}_h, p_h) : [0,T_0] \ra \Sigma_h \times V_h \times W_h$ such that  
\begin{align}
\label{eq:weak-eq1-disc} a_h(\bs{u}_h, \bs{v}) - (p_h, \div \bs{v}) &= (\bs{f}, \bs{v}), & & \bs{v} \in \Sigma_h, \\
\label{eq:weak-eq2-disc} ( \bs{z}_h, \bs{w}) + (p_h, \div \bs{w}) &= 0, & & \bs{w} \in V_h, \\
\label{eq:weak-eq3-disc} (c_0 \dot{p}_h, q) + (\div \dot{\bs{u}}_h, q) - (\div \bs{z}_h, q) &= (g, q), & & q \in W_h,
\end{align}
where 
\begin{align*} 
a_h(\bs{u}, \bs{v}) = 2 \mu (\e_h(\bs{u}), \e_h(\bs{v})) + \lambda (\div \bs{u}, \div \bs{v}), \qquad \bs{u}, \bs{v} \in \Sigma_h.
\end{align*}
For the existence and uniqueness of solutions of (\ref{eq:weak-eq1-disc}--\ref{eq:weak-eq3-disc}), we first point out that this is not a system of ordinary differential equations (ODE) but a differential algebraic equation (DAE). Thus the theory of ODEs for existence of solutions is not available. 

Unfortunately, there is a difficulty adopting a standard DAE theory to prove existence of solutions for the problem.
To see it we review the basic theory of linear DAE problems. Let $\mathbb{C}$ be the complex field and $m$ be a positive integer. For $E_0, E_1 \in \mathbb{C}^{m\times m}$, $F \in C^0([0,\infty); \mathbb{C}^m)$, $X_0 \in \mathbb{C}^m$, a linear DAE is to seek $X \in C^1([0,\infty), \mathbb{C}^m)$ such that 
\begin{align*}
E_0 \dot X + E_1 X = F, \qquad X(0) = X_0.
\end{align*}
In general $E_0$ can be singular. In standard DAE theory, this DAE is called {\it well-posed} if the matrix pencil $E_0 + \lambda E_1$ is regular for some $\lambda \in \mathbb{C}$. However, this definition is easy to mislead because it does {\it not} guarantee the existence of solutions for arbitrary initial data. It is known that a well-posed DAE has a unique solution if given initial data $X_0$ is compatible, i.e., the initial data satisfies some intrinsic algebraic equations of the DAE \cite{DAEbook}. To derive those intrinsic algebraic equations from a DAE, we need to know the Jordan canonical form of the augmented matrix $[E_0\; E_1]$, and it is certainly impractical if $m$ is large. Another difficulty arises from the variance of $c_0$ because finding the Jordan canonical form of the augmented matrix changes drastically when $c_0$ changes. Therefore we will prove the existence and uniqueness of solutions of (\ref{eq:weak-eq1-disc}--\ref{eq:weak-eq3-disc}) directly. 

\begin{thm}
For initial data $p_h(0) \in W_h$ and given $\bs{f} \in C^1([0,T_0]; L^2(\Omega; \R^n))$, $g \in C^0([0,T_0]; L^2(\Omega))$ there exists a unique solution $(\bs{u}_h, \bs{z}_h, p_h) \in C^1([0,T_0]; \Sigma_h \times V_h \times W_h)$ of {\rm (\ref{eq:weak-eq1-disc}--\ref{eq:weak-eq3-disc})}.
\end{thm}
\begin{proof} We show that (\ref{eq:weak-eq1-disc}--\ref{eq:weak-eq3-disc}) is equivalent to an ODE system and check well-posedness of the ODE system. 

Let $\{ \phi_i \}$, $\{ \psi_i \}$, $\{ \chi_i \}$ be bases of $\Sigma_h$, $V_h$, and $W_h$, respectively. We use $\mathbb{A}_{\bs{u} \bs{u}}$, $\mathbb{A}_{\bs{z} \bs{z}}$, $\mathbb{A}_{pp}$, $\mathbb{B}_{\bs{u} p}$, $\mathbb{B}_{\bs{z} p}$ to denote the matrices whose $(i,j)$-entries are
\begin{align*}
a_h(\phi_j, \phi_i), \quad ( \psi_j, \psi_i), \quad (c_0 \chi_j, \chi_i), \quad (\chi_i, \div \phi_j), \quad (\chi_i, \div \psi_j),
\end{align*}
respectively. We write $\bs{u}_h = \sum_i \alpha_i \phi_i$, $\bs{z}_h = \sum_i \beta_i \psi_i$, $p_h = \sum_i \gamma_i \chi_i$, $P_h f = \sum_i \zeta_i \psi_i$, $Q_h g = \sum_i \xi_i \chi_i$ with (time-dependent) coefficient vectors $\bs{\alpha}$, $\bs{\beta}$, $\bs{\gamma}$, $\bs{\zeta}$, $\bs{\xi}$, where $P_h$ is the $L^2$ projection into $V_h$.
Then (\ref{eq:weak-eq1-disc}--\ref{eq:weak-eq3-disc}) can be written as a matrix equation
\begin{align} \label{eq:matrix-eq}
\begin{pmatrix}
0 & 0 & 0 \\
0 & 0 & 0 \\
{\mathbb B}_{\bs{u} p}^T & 0 & {\mathbb A}_{pp}
\end{pmatrix}
\begin{pmatrix}
\dot{\bs{\alpha}} \\
\dot{\bs{\beta}} \\
\dot{\bs{\gamma}}
\end{pmatrix}
+
\begin{pmatrix}
{\mathbb A}_{\bs{u} \bs{u}} & 0 & -{\mathbb B}_{\bs{u} p} \\
0 & {\mathbb A}_{\bs{z} \bs{z}} & {\mathbb B}_{\bs{z} p} \\
0 & -{\mathbb B}_{\bs{z} p}^T & 0
\end{pmatrix}  
\begin{pmatrix}
\bs{\alpha} \\
\bs{\beta} \\
\bs{\gamma}
\end{pmatrix}
= 
\begin{pmatrix}
\bs{\zeta}\\
0 \\
\bs{\xi}
\end{pmatrix}.
\end{align}
It is obvious that ${\mathbb A}_{\bs{z} \bs{z}}$ is symmetric positive definite and ${\mathbb A}_{pp}$ is symmetric positive semidefinite because $c_0 \geq 0$. By Lemma \ref{lemma:korn}, ${\mathbb A}_{\bs{u} \bs{u}}$ is also symmetric positive definite.   

The first and second rows of \eqref{eq:matrix-eq} give
\begin{align} \label{eq:uz-replace}
\bs{\alpha} = {\mathbb A}_{\bs{u} \bs{u}}^{-1} {\mathbb B}_{\bs{u} p} \bs{\gamma} + {\mathbb A}_{\bs{u} \bs{u}}^{-1} \bs{\zeta}, \qquad \bs{\beta} = -{\mathbb A}_{\bs{z} \bs{z}}^{-1} {\mathbb B}_{\bs{z} p} \bs{\gamma}.
\end{align}
The third row of \eqref{eq:matrix-eq} gives ${\mathbb B}_{\bs{u} p}^T \dot{\bs{\alpha}} + {\mathbb A}_{pp} \dot{\bs{\gamma}} - {\mathbb B}_{\bs{z} p}^T \bs{\beta} = \bs{\xi}$. Substituting $\bs{\alpha}$ and $\bs{\beta}$ in this equation using \eqref{eq:uz-replace}, one obtains
\begin{align*}
({\mathbb B}_{\bs{u} p}^T {\mathbb A}_{\bs{u} \bs{u}}^{-1} {\mathbb B}_{\bs{u} p} + {\mathbb A}_{pp}) \dot{\bs{\gamma}}  + {\mathbb B}_{\bs{z} p}^T {\mathbb A}_{\bs{z} \bs{z}}^{-1} {\mathbb B}_{\bs{z} p} \bs{\gamma} = - {\mathbb B}_{\bs{u} p}^T {\mathbb A}_{\bs{u} \bs{u}}^{-1} \dot{\bs{\zeta}} + \bs{\xi}, 
\end{align*}
which is an ODE system of $\bs{\gamma}$. The inf-sup condition \eqref{eq:sigma-property2} implies that the matrix ${\mathbb B}_{\bs{u} p}$ is injective. 
Since ${\mathbb B}_{\bs{u} p}^T {\mathbb A}_{\bs{u} \bs{u}}^{-1} {\mathbb B}_{\bs{u} p}$ is positive definite and ${\mathbb A}_{pp}$ is positive semidefinite, their sum is positive definite. 
Thus the above system has a unique solution $\bs{\gamma}$ for given initial data $\bs{\gamma}(0)$ by a standard ODE theory. Furthermore, ${\mathbb B}_{\bs{u} p}^T {\mathbb A}_{\bs{u} \bs{u}}^{-1} {\mathbb B}_{\bs{u} p} + {\mathbb A}_{pp}$ is still invertible as ${\mathbb A}_{pp} \ra 0$, so this well-posedness is not influenced by small or even vanishing $c_0$. Note that $\bs{\alpha}$, $\bs{\beta}$ are uniquely determined from $\gamma$ by \eqref{eq:uz-replace}. Now the assertion follows from equivalence of \eqref{eq:matrix-eq} and (\ref{eq:weak-eq1-disc}--\ref{eq:weak-eq3-disc}). 
\end{proof}
In this theorem only initial data $p_h (0)$ is given but $\bs{u}_h(0)$, $\bs{z}_h(0)$ are determined by \eqref{eq:weak-eq1-disc} and \eqref{eq:weak-eq2-disc}, which we will call compatibility conditions. 
\begin{defn}
A triple $(\bs{v}', \bs{w}', q') \in \Sigma_h \times V_h \times W_h$ is called a compatible initial data of \rm{(\ref{eq:weak-eq1-disc}--\ref{eq:weak-eq3-disc})} if 
\begin{align}
\label{eq:ic-comp1} a_h(\bs{v}', \bs{v}) - (q', \div \bs{v}) &= (\bs{f}(0), \bs{v}), & & \bs{v} \in \Sigma_h, \\
\label{eq:ic-comp2} ( \bs{w}', \bs{w}) + (q', \div \bs{w}) &= 0, & & \bs{w} \in V_h,
\end{align}
hold.
\end{defn}
If the backward Euler scheme is used for time discretization, then the compatibility of initial data may not be crucial because the compatibility conditions are parts of the equations, so the numerical solution after one time step satisfies the compatibility conditions. However, incompatibility of initial data in DAE problems may generate a spurious numerical solution with the Crank--Nicolson scheme even if it is a very stable time discretization scheme in general. 


\subsection{Error analysis}
We state and prove main results of semidiscrete error analysis. We denote the discrete $H^1$ space with the norm $\| \cdot \|_{1,h}$ by $H_h^1$. 
\begin{thm}\label{thm:semidiscrete-error} Suppose that $(\bs{u}, \bs{z}, p)$ is a solution of \rm{(\ref{eq:new-strong-eq1}--\ref{eq:new-strong-eq3})} which is sufficiently regular  
and $(\bs{u}_h, \bs{z}_h, p_h)$ is a solution of \rm{(\ref{eq:weak-eq1-disc}--\ref{eq:weak-eq3-disc})} with compatible initial data $(\bs{u}_h(0), \bs{z}_h(0), p_h(0))$ such that 
\begin{align} \label{eq:initial-data-approx}
\begin{split}
\| \bs{u}(0) - \bs{u}_h(0) \|_{1,h} + \| p(0) - p_h(0)\|_0 &\lesssim h(\| \bs{u}(0) \|_2 + \| p(0) \|_1), \\
\| \bs{z}(0) - \bs{z}_h(0) \|_0 &\lesssim h\| \bs{z}(0) \|_1.
\end{split}
\end{align}
Then 
\begin{align} \label{eq:semi-total-error}
&\| \bs{u} - \bs{u}_h \|_{L^\infty H_h^1} + \| p - p_h \|_{L^\infty L^2} \\
\notag &\quad  \lesssim h \max \{\| \bs{u} \|_{W^{1,1} H^{2}} + \| p \|_{W^{1,1}H^1} , \| \bs{u} \|_{L^\infty H^2} + \| p \|_{L^2 H^2} + \| p \|_{L^\infty H^1}\}.
\end{align}
and 
\begin{align*}
\| \bs{z} - \bs{z}_h \|_{L^\infty L^2} 
\lesssim 
h \max \{ \| p \|_{W^{1,1} H^2}, \| p \|_{L^\infty H^2} + \| \bs{u} \|_{W^{1,2} H^2} + \| p \|_{W^{1,2}H^1} \}.
\end{align*}
\end{thm}
\begin{rmk} Compatible initial data satisfying \eqref{eq:initial-data-approx} can be found by 
setting $p_h(0) = Q_h p(0)$ and finding $\bs{u}_h(0)$ and $\bs{z}_h(0)$ using \rm{(\ref{eq:ic-comp1}--\ref{eq:ic-comp2})}. 
\end{rmk}
For the proof of Theorem \ref{thm:semidiscrete-error} we need some preliminary results.

\begin{lemma} \label{lemma:consistency-err-bound}
For $\utilde{\bs{\tau}} \in H^1(\Omega; \R^{n \times n})$ and $\bs{v} \in \Sigma_h$ define 
\begin{align*}
E_h (\utilde{\bs{\tau}}, \bs{v} ) := \sum_{E \in \mathcal{E}_h} \langle \utilde{\bs{\tau}} \bs{n}_E, \lb \bs{v} \rb \rangle_E .
\end{align*}
Then 
\begin{align*}
| E_h (\utilde{\bs{\tau}}, \bs{v}) | \lesssim h \| \utilde{\bs{\tau}} \|_1 \| \bs{v} \|_{1,h}. 
\end{align*}
\end{lemma}
\begin{proof}
Let $E \in \mathcal{E}_h$ be an interior edge/face in $\Omega$ and $T_+$, $T_-$ be two distinct triangles/tetrahedra sharing $E$ as the common boundary. The shape regularity of $\mathcal{T}_h$ yields $h_{T_-} \sim h_{T_+} \sim h_E$ where $h_{T_-}$, $h_{T_+}$, $h_E$ are the diameters of $T_-$, $T_+$, $E$, respectively. 

Since $\lb \bs{v} \rb|_E$ is perpendicular to $\mathcal{P}_0(E; \R^n)$, 
\begin{align*}
\langle \utilde{\bs{\tau}} \bs{n}_E, \lb \bs{v} \rb \rangle_E =  \langle \utilde{\bs{\tau}} \bs{n}_E - \bs{c}, \lb \bs{v} \rb \rangle_E \leq \inf_{\bs{c} \in \R^n} \| \utilde{\bs{\tau}} \bs{n}_E - \bs{c} \|_{0,E} \| \lb \bs{v} \rb \|_{0,E}. 
\end{align*}
By a standard scaling argument and shape regularity of meshes one can see that 
\begin{align*}
\| \lb \bs{v} \rb \|_{0,E} &\lesssim h_E^{\half} \| \nabla \bs{v} \|_{0, T_+ \cup T_-}, \\
\inf_{\bs{c} \in \R^n} \| \utilde{\bs{\tau}} \bs{n}_E - \bs{c} \|_{0,E} &\lesssim h_E^{\half} \| \nabla \bs{\tau} \|_{0, T_+ \cup T_-}. 
\end{align*}
For an edge/face $E$ on boundary and the triangle $T \in \mathcal{T}_h$ containing $E$ in its boundary, a similar argument gives 
\begin{align*}
\| \lb \bs{v} \rb \|_{0,E} &\lesssim h_E^{\half} \| \nabla \bs{v} \|_{0, T}, \\
\inf_{\bs{c} \in \R^n} \| \utilde{\bs{\tau}} \bs{n}_E - \bs{c} \|_{0,E} &\lesssim h_E^{\half} \| \nabla \bs{\tau} \|_{0, T}. 
\end{align*}
By the Cauchy--Schwarz inequality and the above results, we have 
\begin{align*}
E_h(\utilde{\bs{\tau}}, \bs{v}) \lesssim h \| \utilde{\bs{\tau}} \|_1 \| \bs{v} \|_{1,h} .
\end{align*} 
as desired. 
\end{proof}

The following simple lemma will be useful in our error analysis.

\begin{lemma} \label{lemma:basic-ineq}
Suppose that $A, B, C, D >0$ satisfy 
\begin{align*}
A^2 + B^2 \leq CA + D. 
\end{align*}
Then either $A + B \leq 4C$ or $A + B \leq 2\sqrt{D}$ holds.
\end{lemma}
\begin{proof}
Since either $CA \leq D$ or $CA \geq D$ is true, one of the followings holds.
\begin{align}
\label{eq:two-ineq} A^2 + B^2 \leq 2D, \qquad A^2 + B^2 \leq 2CA.
\end{align}

If the first inequality in \eqref{eq:two-ineq} holds, then 
\begin{align*}
(A+B)^2 \leq 2(A^2 + B^2) \leq 4D, 
\end{align*}
which implies $A + B \leq 2\sqrt{D}$. 

Suppose that the second inequality in \eqref{eq:two-ineq} holds. If $B \geq A$, then dividing the inequality by $A$ gives 
\begin{align*}
A + B \leq (A^2 + B^2)/A \leq 2C.
\end{align*}
If $B \leq A$, then dividing the second inequality in \eqref{eq:two-ineq} by $A$ gives $A \leq 2C$, and therefore we have $A + B \leq 2A \leq 4C$ as desired. 
\end{proof}

Now we are ready to prove Theorem \ref{thm:semidiscrete-error}. 

\begin{proof}[of Theorem \ref{thm:semidiscrete-error}] We denote the errors by  
\begin{align*}
e_{\bs{u}} = \bs{u} - \bs{u}_h, \quad e_{\bs{z}} = \bs{z} - \bs{z}_h, \quad e_p = p - p_h.
\end{align*}
From the left-hand side of \eqref{eq:new-strong-eq1}, through the integration by parts, we obtain 
\begin{align*} 
-(\div \mathcal{C} \e(\bs{u}) + \nabla p, \bs{v}) = (\mathcal{C} \e(\bs{u}), \e_h(\bs{v})) - (p, \div \bs{v}) + E_h(\utilde{\bs{\sigma}} , \bs{v}), \quad \bs{v} \in \Sigma_h.
\end{align*}
The difference of this and \eqref{eq:weak-eq1-disc} gives
\begin{align} 
\label{eq:err-eq1} a_h(e_{\bs{u}}, \bs{v}) - (e_p, \div \bs{v}) = - E_h(\utilde{\bs{\sigma}} , \bs{v}), & & \bs{v} \in \Sigma_h.
\end{align}
In addition, the differences of (\ref{eq:weak-eq2}--\ref{eq:weak-eq3}) and (\ref{eq:weak-eq2-disc}--\ref{eq:weak-eq3-disc}) yield
\begin{align}
\label{eq:err-eq2} ( e_{\bs{z}}, \bs{w}) + (e_p, \div \bs{w}) &= 0, & & \bs{w} \in V_h, \\
\label{eq:err-eq3} (c_0 \dot{e}_p, q) + (\div \dot{e}_{\bs{u}}, q) - (\div e_{\bs{z}}, q) &= 0, & & q \in W_h.
\end{align}
For the error analysis we split the errors into 
\begin{align}
\label{eq:err-split1} e_{\bs{u}} = e_{\bs{u}}^I + e_{\bs{u}}^A &:= (\bs{u} - \Pi_h \bs{u}) + ( \Pi_h \bs{u} - \bs{u}_h), \\
\label{eq:err-split2} e_{\bs{z}} = e_{\bs{z}}^I + e_{\bs{z}}^A &:= (\bs{z} - \Pi_h^{RT} \bs{z}) + (\Pi_h^{RT} \bs{z} - \bs{z}_h), \\
\label{eq:err-split3} e_p = e_p^I + e_p^A &:= (p - Q_h p) + (Q_h p - p_h).
\end{align}
By \eqref{eq:approx1} and \eqref{eq:sigma-property2}
\begin{align} \label{eq:interp-approx}
\| e_{\bs{u}}^I (t) \|_{1,h} \lesssim h \| \bs{u} (t) \|_2, \quad \| e_{\bs{z}}^I (t) \|_0 \lesssim h \| \bs{z} (t) \|_1, \quad \| e_p^I (t) \|_0 \lesssim h \| p(t) \|_1.
\end{align}
Similar inequalities also hold for the time derivatives of $e_{\bs{u}}^I$, $e_{\bs{z}}^I$, $e_p^I$. These inequalities, \eqref{eq:initial-data-approx}, and the triangle inequality give
\begin{align} \label{eq:initial-data-approx2}
\begin{split}
\| e_{\bs{u}}^A (0) \|_{1,h} + \| e_p^A(0) \|_0 &\lesssim h (\| \bs{u}(0) \|_2 + \| p(0) \|_1) \\
\| e_{\bs{z}}^A (0) \|_0 &\lesssim h\| \bs{z}(0) \|_1. 
 \end{split}
\end{align}
Furthermore, by the definitions of $e_{\bs{u}}^I$, $e_{\bs{z}}^I$, $e_p^I$ and the properties in \eqref{eq:commute1}, 
\begin{gather}
\label{eq:err-cancel1} \div \dot{e}_{\bs{u}}^I \perp W_h, \qquad \div e_{\bs{z}}^I \perp W_h, \qquad e_p^I \perp \div \Sigma_h. 
\end{gather}
Rewriting (\ref{eq:err-eq1}--\ref{eq:err-eq3}) using (\ref{eq:err-split1}--\ref{eq:err-split3}) and the orthogonalities in \eqref{eq:err-cancel1}, we have 
\begin{align}
\label{eq:new-err-eq1} a_h(e_{\bs{u}}^A, \bs{v}) - (e_p^A, \div \bs{v}) &= - a_h(e_{\bs{u}}^I, \bs{v}) - E_h(\utilde{\bs{\sigma}} , \bs{v}), \\
\label{eq:new-err-eq2} ( e_{\bs{z}}^A, \bs{w}) + (e_p^A, \div \bs{w}) &= - ( e_{\bs{z}}^I, \bs{w}), \\
\label{eq:new-err-eq3} (c_0 \dot{e}_p^A, q) + (\div \dot{e}_{\bs{u}}^A, q) - (\div e_{\bs{z}}^A, q) &= - (c_0 \dot{e}_p^I, q) , 
\end{align}
for any $(\bs{v}, \bs{w}, q) \in \Sigma_h \times V_h \times W_h$. The rest of proof consists of three parts: estimates of $\| \bs{u} - \bs{u}_h \|_{L^\infty H_h^1}$, $\| \bs{z} - \bs{z}_h \|_{L^\infty L^2}$, and $\| p - p_h \|_{L^\infty L^2}$. 

{\bf Estimate of $\| \bs{u} - \bs{u}_h \|_{L^\infty H_h^1}$} :
Let $X(t), Y(t)$ be 
\begin{align*}
X^2(t) = \| e_{\bs{u}}^A (t) \|_{a,h}^2 + \| e_p^A (t) \|_{c_0}^2, \qquad 
Y^2(t) = \int_0^t \| e_{\bs{z}}^A \|_0 \,ds,
\end{align*}
where $\| \bs{v} \|_{a,h}^2 := a_h(\bs{v}, \bs{v})$ and $\| q \|_{c_0}^2 := (c_0 q, q)$. Since $\| \cdot \|_{1,h} \sim \| \cdot \|_{a,h}$, we will estimate $\max_{0 \leq t \leq T_0} X(t)$ instead of $\max_{0\leq t \leq T_0} \| e_{\bs{u}}^A (t)\|$. 

If we take $\bs{v} = \dot{e}_{\bs{u}}^A$, $\bs{w} = e_{\bs{z}}^A$, $q = e_p^A$ in (\ref{eq:new-err-eq1}--\ref{eq:new-err-eq3}), and add those equations together, then 
\begin{align*}
\half \frac{d}{dt} (X(t))^2 + \| e_{\bs{z}}^A (t) \|_0^2 = - a_h(e_{\bs{u}}^I, \dot{e}_{\bs{u}}^A ) - E_h(\utilde{\bs{\sigma}} , \dot{e}_{\bs{u}}^A ) - ( e_{\bs{z}}^I, e_{\bs{z}}^A) - (c_0 \dot{e}_p^I, e_p^A). 
\end{align*}
Integrating this from $0$ to $t$ in time,
\begin{align}
\label{eq:interm-err-eq1}& (X(t))^2 + 2(Y(t))^2 \\
\notag & = (X(0))^2 + 2\int_0^t \left ( - a_h(e_{\bs{u}}^I, \dot{e}_{\bs{u}}^A ) - E_h(\utilde{\bs{\sigma}} , \dot{e}_{\bs{u}}^A) - (e_{\bs{z}}^I, e_{\bs{z}}^A) - (c_0 \dot{e}_p^I, e_p^A) \right) \,ds, \\
\notag & =: (X(0))^2 + \Phi_1(t) + \Phi_2(t) + \Phi_3(t) + \Phi_4(t). 
\end{align}
Defining $\bar t$ by
\begin{align*} 
X(\bar{t}) = \sup_{0 \leq t \leq T_0} X(t),
\end{align*}
it suffices to estimate $X(\bar t)$. Moreover, since 
\begin{align*}
\text{either }\quad \Phi_1(\bar t) + \Phi_2(\bar t) + \Phi_4(\bar t) \leq \Phi_3(\bar t)  \qquad \text{or} \qquad  \Phi_1(\bar t) + \Phi_2(\bar t) + \Phi_4(\bar t) \geq \Phi_3(\bar t) ,
\end{align*}
is always true, we have, from \eqref{eq:interm-err-eq1}, that
\begin{align}
\label{eq:estm-case1} &\text{either} & (X({\bar t}))^2 + 2(Y({\bar t}))^2 &\leq (X(0))^2 + 2 \Phi_3(\bar t), \\
\label{eq:estm-case2} &\text{or } & (X({\bar t}))^2 + 2(Y({\bar t}))^2 &\leq (X(0))^2 + 2(\Phi_1(\bar t) + \Phi_2(\bar t) + \Phi_4(\bar t) ).
\end{align}
Thus it is enough to prove an estimate of $X(\bar{t})$ for these two cases. 


{\it Case I} : Suppose that \eqref{eq:estm-case1} is true. Note that 
\begin{align*} 
2\Phi_2(\bar t) = -4\int_0^{\bar{t}} ( e_{\bs{z}}^I, e_{\bs{z}}^A) \,ds\leq 2 \int_0^{\bar{t}} \| e_{\bs{z}}^I \|_0^2 \,ds + 2(Y(\bar{t}))^2,
\end{align*}
by Young's inequality and the definition of $Y(\bar t)$. Combining it with \eqref{eq:estm-case1} gives 
\begin{align*}
X^2(\bar{t}) &\leq X^2(0) + 2 \int_0^{\bar{t}} \| e_{\bs{z}}^I \|_0^2 \,ds 
\leq X(0)^2 + ch^2 \| \bs{z} \|_{L^2 H^1}^2. 
\end{align*}
Estimating $X(0)$ using \eqref{eq:initial-data-approx2} and taking square roots, 
\begin{align} \label{eq:semi-x-estm1}
X(\bar t) &\lesssim h (\| \bs{u}(0) \|_2  + \| p(0) \|_1 + \| \bs{z} \|_{L^2 H^1}) \\
\notag & \lesssim h ( \| \bs{u} (0) \|_2 + \| p(0) \|_1 + \| p \|_{L^2 H^2}).
\end{align}

{\it Case II} : Suppose that \eqref{eq:estm-case2} is true. To estimate $\Phi_1(\bar t)$, we use the integration by parts in time, which gives
\begin{align} \label{eq:semi-Phi1-estm} 
\Phi_1(\bar t) &= -2 \int_0^{\bar t} a_h(e_{\bs{u}}^I, \dot{e}_{\bs{u}}^A) \, ds , \\
\notag &= 2 a_h(e_{\bs{u}}^I(\bar t), e_{\bs{u}}^A(\bar t)) - 2 a_h(e_{\bs{u}}^I(0), e_{\bs{u}}^A(0)) + 2 \int_0^{\bar t}  a_h(\dot e_{\bs{u}}^I , e_{\bs{u}}^A) \, ds \\
\notag &\leq 2\left( \| e_{\bs{u}}^I (\bar t) \|_{a,h} + \| e_{\bs{u}}^I (0) \|_{a,h} + \int_0^{\bar t} \| \dot e_{\bs{u}}^I (0) \|_{a,h} \,ds \right) X(\bar t) \\
\notag &\leq ch \| \bs{u} \|_{W^{1,1}H^2} X(\bar t). 
\end{align}
Similarly, 
\begin{align} \label{eq:semi-Phi2-estm} 
\Phi_2(\bar t) &= -2\int_0^{\bar t}  E_h(\utilde{\bs{\sigma}}, \dot{e}_{\bs{u}}^A )  \, ds \\
\notag &= 2\left( - E_h(\utilde{\bs{\sigma}}(\bar t) , e_{\bs{u}}^A(\bar t)) + E_h( \utilde{\bs{\sigma}}(0) , e_{\bs{u}}^A(0)) \right) + 2\int_0^{\bar t}  E_h(\dot{\utilde{\bs{\sigma}}} , e_{\bs{u}}^A) \, ds \\
\notag &\leq ch \left( \| \utilde{\bs{\sigma}}(\bar t) \|_1 + \| \utilde{\bs{\sigma}}(0) \|_1 + \int_0^{\bar t} \| \dot{\utilde{\bs{\sigma}}} \|_1 \,ds \right) X(\bar t) \\
\notag &\leq ch \| \utilde{\bs{\sigma}} \|_{W^{1,1}H^1} X(\bar t).  
\end{align}
For $\Phi_4(\bar t)$, 
\begin{align}
\label{eq:semi-Phi4-estm} 
\Phi_4 (\bar t) = -2 \int_0^{\bar t} (\dot{e}_p^I, e_p^A)_{c_0} \,ds \leq 
2  \int_0^{\bar t} \| \dot{e}_p^I \|_{c_0} \,ds  X(\bar{t}) \leq ch \| p \|_{W^{1,1}H^1} X(\bar t).  
\end{align}
Combining \eqref{eq:estm-case2}, \eqref{eq:semi-Phi1-estm}, \eqref{eq:semi-Phi2-estm}, and \eqref{eq:semi-Phi4-estm}, we have an inequality of the form in Lemma \ref{lemma:basic-ineq} with $A = X(\bar t)$, $B = \sqrt{2} Y(\bar t)$, $D = X(0)^2$, and 
\begin{align*}
C = ch  (\| \utilde{\bs{\sigma}}, p \|_{W^{1,1}H^1} + \| \bs{u} \|_{W^{1,1}H^2}) . 
\end{align*}
Note that $\| \utilde{\bs{\sigma}} \|_{W^{1,1}H^1} \lesssim \| \bs{u} \|_{W^{1,1} H^2} + \| p \|_{W^{1,1} H^1}$ because $\utilde{\bs{\sigma}} = \mathcal{C}\e(\bs{u}) - p \utilde{\bs{I}}$. By Lemma \ref{lemma:basic-ineq} and \eqref{eq:initial-data-approx2} 
\begin{align}
\label{eq:semi-x-estm2}
\begin{split} 
&X(\bar t) \lesssim ch (\| \utilde{\bs{\sigma}}, p \|_{W^{1,1} H^1} + \| \bs{u} \|_{W^{1,1}H^2}) \lesssim h (\| p \|_{W^{1,1} H^1} + \| \bs{u} \|_{W^{1,1}H^2}) \\
\text{ or } \; &X(\bar t)  \lesssim h (\| \bs{u}(0) \|_2 + \| p(0) \|_1 ) .
\end{split} 
\end{align}
Now we complete the proof of the estimate of $\| \bs{u} - \bs{u}_h \|_{L^{\infty} H_h^1}$.  By the triangle inequality and the inequality $\| e_{\bs{u}}^A(t) \|_{1,h} \lesssim X(\bar t)$, 
\begin{align*}
\| \bs{u}(t) - \bs{u}_h (t) \|_{1,h} \leq \| e_{\bs{u}}^I (t) \|_{1,h} + \| e_{\bs{u}}^A (t) \|_{1,h} \lesssim \| e_{\bs{u}}^I (t) \|_{1,h} + X(\bar t) .
\end{align*}
By \eqref{eq:interp-approx}, the triangle inequality, \eqref{eq:semi-x-estm1}, \eqref{eq:semi-x-estm2}, and \eqref{eq:sobolev},
\begin{multline*} 
\| \bs{u}(t) - \bs{u}_h (t) \|_{1,h} \\
\lesssim h \max \{\| \bs{u} \|_{W^{1,1} H^2} + \| p \|_{W^{1,1}H^1}, \| \bs{u} \|_{L^\infty H^2} + \| p \|_{L^2 H^2} + \| p \|_{L^\infty H^1}\},
\end{multline*}
for any $0 \leq t \leq T_0$, which is the estimate for $\| \bs{u} - \bs{u}_h \|_{L^\infty H_h^1}$ in \eqref{eq:semi-total-error}.

{\bf Estimate of $\| p - p_h \|_{L^\infty L^2}$} : 
By the inf-sup condition \eqref{eq:inf-sup1}, for any $0 \not = q \in W_h$, there exists a $\bs{v} \in \Sigma_h$ such that 
\begin{align*}
(\div \bs{v}, q') = (q, q'), \quad \forall q' \in W_h, \qquad \| \bs{v} \|_{1,h} \lesssim \| q \|_0. 
\end{align*}
If we use this $\bs{v}$ in \eqref{eq:new-err-eq1} with $q = e_p^A(t)$, then 
\begin{align*}
\| e_p^A (t) \|_0^2 &= a_h(\bs{u}(t) - \bs{u}_h(t), \bs{v}) + E_h(\utilde{\bs{\sigma}}(t) , \bs{v}) \\
&\lesssim (\| \bs{u} (t) - \bs{u}_h (t) \|_{1,h} + h\| \utilde{\bs{\sigma}} (t) \|_1) \| \bs{v} \|_{1,h} ,
\end{align*}
where the last inequality is due to the estimate of $\| \bs{u} - \bs{u}_h \|_{L^\infty H_h^1}$. By the triangle inequality, the above estimate, \eqref{eq:interp-approx}, and the estimate of $\| \bs{u} - \bs{u}_h \|_{L^\infty H_h^1}$, 
\begin{multline*}
\| p - p_h \|_{L^{\infty} L^2} \\
\lesssim h \max \{ \| \bs{u} \|_{W^{1,1} H^2} +\| p \|_{W^{1,1} H^1}, \| \bs{u} \|_{L^\infty H^2} + \| p \|_{L^2 H^2} + \| p \|_{L^{\infty}H^1}\},
\end{multline*}
so \eqref{eq:semi-total-error} is proven.

{\bf Estimate of $\| \bs{z} - \bs{z}_h \|_{L^\infty L^2}$} : 
The time derivatives of the equations (\ref{eq:new-err-eq1}--\ref{eq:new-err-eq2}) give 
\begin{align*}
a_h(\dot e_{\bs{u}}^A, \bs{v}) - (\dot e_p^A, \div \bs{v}) &= - a_h(\dot e_{\bs{u}}^I, \bs{v}) - E_h(\dot{\utilde{\bs{\sigma}}} , \bs{v}), & & \bs{v} \in \Sigma_h, \\
(\dot e_{\bs{z}}^A, \bs{w}) + (\dot e_p^A, \div \bs{w}) &= - ( \dot e_{\bs{z}}^I, \bs{w}), & & \bs{w} \in V_h, 
\end{align*}
Taking $\bs{v} = \dot e_{\bs{u}}^A$, $\bs{w} = e_{\bs{z}}^A$ in the above, taking $q = \dot e_p^A$ in \eqref{eq:new-err-eq3}, and adding these three equations together yield
\begin{align*}
\half \frac{d}{dt} \| e_{\bs{z}}^A \|_0^2 + \| \dot e_{\bs{u}}^A \|_{a,h}^2 + \| \dot e_p^A \|_{c_0}^2 
&= - a_h(\dot e_{\bs{u}}^I, \dot e_{\bs{u}}^A) - E_h(\dot{\utilde{\bs{\sigma}}} , \dot e_{\bs{u}}^A) - ( \dot e_{\bs{z}}^I, e_{\bs{z}}^A) \\
&\quad + (c_0 \dot e_p^I, \dot e_p^A).
\end{align*}
Let $\| e_{\bs{z}}^A(\bar t) \|_0 = \max_{0\leq t \leq T_0} \| e_{\bs{z}}^A(t) \|_0$. By integrating the above from 0 to $\bar t$,
\begin{align*}
&\| e_{\bs{z}}^A (\bar t) \|_0^2 + \int_0^{\bar t} \{ \| \dot e_{\bs{u}}^A \|_{a,h}^2 + \| \dot e_p^A \|_{c_0}^2 \} ds \\
\notag &= \| e_{\bs{z}}^A (0) \|_0^2 + \int_0^{\bar t} \{ - a_h(\dot e_{\bs{u}}^I, \dot e_{\bs{u}}^A) - E_h(\dot{\utilde{\bs{\sigma}}} , \dot e_{\bs{u}}^A) - ( \dot e_{\bs{z}}^I, e_{\bs{z}}^A) + (c_0 \dot e_p^I, \dot e_p^A)  \} ds \\
\notag &=: \| e_{\bs{z}}^A (0) \|_0^2 + \Psi_1(\bar t) + \Psi_2(\bar t) + \Psi_3(\bar t) + \Psi_4(\bar t).
\end{align*}

If $\Psi_1(\bar t) + \Psi_2(\bar t) + \Psi_4(\bar t) \leq \Psi_3(\bar t)$, then 
\begin{align*}
\| e_{\bs{z}}^A (\bar t) \|_0^2 \leq \| e_{\bs{z}}^A (0) \|_0^2 + 2 \Psi_3(\bar t).
\end{align*}
Using the Cauchy--Schwarz inequality and dividing both sides by $\| e_{\bs{z}}^A (\bar t) \|_0$ yield
\begin{align*} 
\| e_{\bs{z}}^A (\bar t) \|_0 &\leq \| e_{\bs{z}}^A (0) \|_0 + 2 \int_0^{\bar t} \| \dot e_{\bs{z}}^I \|_0 ds .
\end{align*}
If we use \eqref{eq:initial-data-approx2} and \eqref{eq:interp-approx} to estimate $\| e_{\bs{z}}^A(0) \|_0$ and the integral term, then 
\begin{align} \label{eq:semi-z-estm1}
\| e_{\bs{z}}^A (\bar t) \|_0 \lesssim h( \| \bs{z}(0) \|_1 + \| \bs{z} \|_{W^{1,1} H^1}) \lesssim h \|\bs{z} \|_{W^{1,1}H^1},
\end{align}
where the last inequality is due to \eqref{eq:sobolev}. 

On the other hand, if $\Psi_1(\bar t) + \Psi_2(\bar t) + \Psi_4(\bar t) \geq \Psi_3(\bar t)$, then 
\begin{align*}
\| e_{\bs{z}}^A (\bar t) \|_0^2 + \int_0^{\bar t} \{ \| \dot e_{\bs{u}}^A \|_{a,h}^2 + \| \dot e_p^A \|_{c_0}^2 \} ds \leq \| e_{\bs{z}}^A (0) \|_0^2 + 2 (\Psi_1(\bar t) + \Psi_2(\bar t) + \Psi_4(\bar t)).
\end{align*}
The Cauchy--Schwarz and Young's inequalities yield
\begin{align} \label{eq:semi-z-estm2}
\| e_{\bs{z}}^A (\bar t) \|_0^2 &\leq \| e_{\bs{z}}^A (0) \|_0^2 + ch^2 (\| \bs{u} \|_{W^{1,2}([0,\bar{t}];H^2)}^2 + \| \utilde{\bs{\sigma}}, p \|_{W^{1,2}([0,\bar{t}]; H^1)}^2) \\
\notag &\lesssim h^2( \| \bs{z}(0) \|_1^2  + \| \bs{u} \|_{W^{1,2} H^2}^2 + \| p \|_{W^{1,2}H^1}^2).
\end{align}
By the triangle inequality, \eqref{eq:interp-approx} and estimates \eqref{eq:semi-z-estm1}, \eqref{eq:semi-z-estm2}, 
\begin{align*}
\| \bs{z} - \bs{z}_h \|_{L^{\infty}L^2} &\lesssim \| e_{\bs{z}}^I \|_{L^\infty L^2} + \| e_{\bs{z}}^A \|_{L^{\infty}L^2} \\
&\lesssim h \max \{ \| \bs{z} \|_{W^{1,1} H^1}, \| \bs{z} \|_{L^\infty H^1} + \| \bs{u} \|_{W^{1,2} H^2} + \| p \|_{W^{1,2}H^1} \}.
\end{align*}

\end{proof}

\section{Error analysis of fully discrete solutions} \label{sec:fullydiscrete-analysis}
In this section we consider the error analysis of the fully discrete solutions with the backward Euler time discretization. As in our error analysis for the semidiscrete solutions we do not use Gr\"{o}nwall's inequality, so our error bounds do not contain exponentially growing factors. 

\subsection{Fully-discrete problem and its well-posedness}
Let $\lap t >0$ be the time step size such that $T_0 = N \lap t$ for an integer $N$,  and $t_j = j \lap t$ for $j = 0, 1, \cdots, N$. For a continuous function $f$ on $[0,T_0]$, we define $f^j = f(t_j)$. 
For a sequence $\{f^j\}_{j \geq 0}$, define
\begin{align}
\label{time-quotient}
\pd_t f^{j+1} &= \frac{f^{j+1} - f^j}{\lap t}. 
\end{align}
Suppose that $(\bs{U}^{0}, \bs{Z}^{0}, P^{0})$ is a numerical initial data satisfying \eqref{eq:initial-data-approx} but not necessarily compatible. In the backward Euler scheme, $(\bs{U}^{j+1}, \bs{Z}^{j+1}, P^{j+1})$, the numerical solution at the $(j+1)$-th time step is defined inductively by
\begin{align}
\label{eq:fulldisc1} a_h( \bs{U}^{j+1} , \bs{v} ) - ( P^{j+1}, \div \bs{v}) &= ( \bs{f}^{j+1}, \bs{v} ) , \\
\label{eq:fulldisc2} (\bs{Z}^{j+1}, \bs{w}) + ( P^{j+1}, \div \bs{w}) &= 0 , \\
\label{eq:fulldisc3} ( c_0 \pd_t  P^{j+1}, q ) + ( \div ( \pd_t \bs{U}^{j+1}) , q ) - (\div \bs{Z}^{j+1}, q) &= ( g^{j+1}, q ) ,
\end{align}
for $(\bs{v}, \bs{w}, q) \in \Sigma_h \times V_h \times W_h$ and $j \geq 0$.

In order to prove that the fully discrete solution is well-defined we show that $(\bs{U}^{j+1}, \bs{Z}^{j+1}, P^{j+1})$ is uniquely determined by the linear system (\ref{eq:fulldisc1}--\ref{eq:fulldisc3}) when $\bs{U}^{j}, P^{j}, \bs{f}^{j+1}, g^{j+1}$ are given. Rewriting (\ref{eq:fulldisc1}--\ref{eq:fulldisc3}),
\begin{align*}
a_h( \bs{U}^{j+1} , \bs{v} ) - (P^{j+1}, \div \bs{v}) &= ( \bs{f}^{j+1}, \bs{v} ) , \\
(\bs{Z}^{j+1}, \bs{w}) + (P^{j+1}, \div \bs{w}) &= 0 , \\
( c_0 P^{j+1}, q ) + ( \div \bs{U}^{j+1} , q ) - \lap t (\div \bs{Z}^{j+1}, q) &= ( c_0 P^{j}, q ) + ( \div \bs{U}^{j} , q ) \\
&\qquad + \lap t ( g^{j+1}, q ) , 
\end{align*}
for $(\bs{v}, \bs{w}, q) \in \Sigma_h \times V_h \times W_h$. Regarding $\bs{U}^{j+1}, \bs{Z}^{j+1}, P^{j+1}$ as unknowns, the above is a system of linear equations with the same number of equations and unknowns. Suppose that $\bs{U}^{j} = \bs{Z}^j = P^{j} = \bs{f}^{j+1} = g^{j+1} = 0$ and we want to show that $\bs{U}^{j+1}=\bs{Z}^{j+1}=P^{j+1}=0$. If we take $\bs{v} = \bs{U}^{j+1}$, $\bs{w} = \bs{Z}^{j+1}$, $q = P^{j+1}$ and add all equations together, then we have 
\begin{align*}
0 \leq a_h(\bs{U}^{j+1}, \bs{U}^{j+1}) + (\bs{Z}^{j+1}, \bs{Z}^{j+1}) + (c_0 P^{j+1}, P^{j+1}) = 0,
\end{align*}
which implies $\bs{U}^{j+1} = \bs{Z}^{j+1} = 0$. Note that \eqref{eq:fulldisc1} is now given as $(P^{j+1}, \div \bs{v}) = 0$ for any $\bs{v} \in \Sigma_h$. Then $P^{j+1} = 0$ because $\div \Sigma_h = W_h$. Hence the fully discrete solution is well-defined.



\subsection{Error analysis} Let us split the error $(\bs{u}^j - \bs{U}^j, \bs{z}^j - \bs{Z}^j,  p^j - P^j)$ as
\begin{align}
\label{eq:err-decomp1} \bs{u}^j - \bs{U}^j &= (\bs{u}^j - \Pi_h \bs{u}^j) + (\Pi_h \bs{u}^j - \bs{u}^j) =: e_{\bs{u}}^{I,j} + \theta_{\bs{u}}^j , \\
\label{eq:err-decomp2} \bs{z}^j - \bs{Z}^j &= (\bs{z}^j - \Pi_h^{RT} \bs{z}^j) + (\Pi_h^{RT} \bs{z}^j - \bs{Z}^j) =: e_{\bs{z}}^{I,j} + \tv^j ,\\
\label{eq:err-decomp3} p^j - P^j &= (p^j - Q_h p^j) + (Q_h p^j - P^j) =: e_p^{I,j} + \tp^j .
\end{align}
Applying a similar argument used to obtain (\ref{eq:new-err-eq1}--\ref{eq:new-err-eq3}) yields
\begin{align} 
\label{eq:theta-j-eq1} a_h({\theta}_{\bs{u}}^{j+1}, \bs{v}) - ({\theta}_p^{j+1}, \div \bs{v}) &= -a_h({e}_{\bs{u}}^{I,j+1}, \bs{v}) - E_h( {\utilde{\bs{\sigma}}}^{j+1} , \bs{v} ),  \\
\label{eq:theta-j-eq2} ( \theta_{\bs{z}}^{j+1}, \bs{w}) + ( \theta_p^{j+1}, \div \bs{w}) &= - ( e_{\bs{z}}^{I, j+1}, \bs{w}), \\
\label{eq:theta-j-eq3} (c_0 \pd_t \theta_p^{j+1}, q) + (\div \pd_t \theta_{\bs{u}}^{j+1}, q) - (\div {\theta}_{\bs{z}}^{j+1}, q) &= (\omega_1^{j+1} + \omega_2^{j+1} , q), 
\end{align}
where 
\begin{align*} 
\omega_1^{j+1} &= c_0 (\bar \pd_t p^{j+1} - {\dot p}^{j+1} -\bar \pd_t e_p^{I, j+1}) , \\
\omega_2^{j+1} &= \div (\bar \pd_t \bs{u}^{j+1} - {\dot{\bs{u}}}^{j+1} - \bar \pd_t e_{\bs{u}}^{I,j+1}).
\end{align*}

\begin{thm}
\label{thm:full-error-estimate} Suppose that $(\bs{u}, \bs{z},p)$ is an exact solution of \rm{(\ref{eq:new-strong-eq1}--\ref{eq:new-strong-eq3})} with sufficient regularity and a fully discrete solution $\{(\bs{U}^j, \bs{Z}^j, P^j)\}_{1 \leq j \leq N}$ is defined by {\rm (\ref{eq:fulldisc1}--\ref{eq:fulldisc3})} with initial data satisfying \eqref{eq:initial-data-approx}. Let $M_1$ be the maximum of 
$\| \bs{u} \|_{W^{2,1}H^{1} \cap W^{1,1}H^2} + \| p \|_{W^{2,1}L^2 \cap W^{1,1}H^1}$ and 
$\| \bs{u} \|_{L^\infty H^2} + \| p \|_{W^{1,2}H^2 \cap L^\infty H^1}$ and $M_2$ be the maximum of $\| \bs{u} \|_{W^{1,2}H^2 \cap W^{2,2}H^1} + \| p \|_{W^{1,2}H^1 \cap W^{2,2}L^2}$ and $\| p \|_{W^{1,1}H^2}$. Then 
\begin{align*}
\max_{1 \leq i \leq N} \| \bs{u}^i - \bs{U}^i \|_{1,h} + \max_{1 \leq i \leq N} \| p^i - P^i \|_0 &\leq c(\lap t + h) M_1, \\
\max_{1 \leq i \leq N} \| \bs{z}^i - \bs{Z}^i \|_{0} &\leq c(\lap t + h) M_2,
\end{align*}
with constants $c$ independent of $T_0$. 
\end{thm}
We already have error bounds of $\| e_{\bs{u}}^{I,j}\|_{1,h}$, $\| e_{\bs{z}}^{I,j} \|_0$, and $\| e_p^{I,j} \|_0$ in \eqref{eq:interp-approx}. Thus, by the triangle inequality, we only need to estimate $\| \theta_{\bs{u}}^j \|_{1,h}$, $\| \theta_{\bs{z}}^j \|_0$, and $\| \theta_p^j \|_0$ for the proof of Theorem \ref{thm:full-error-estimate}, therefore we will devote the rest of this section to prove these estimates.

\begin{lemma}
For $\theta_{\bs{u}}^j$, $\theta_p^j$ defined in \eqref{eq:err-decomp1} and \eqref{eq:err-decomp3}, the following hold.
\begin{align}
\label{eq:thetap-estm1} \| \theta_p^j \|_0 &\lesssim \| \theta_{\bs{u}}^j \|_{1,h} + \| e_{\bs{u}}^{I, j} \|_{1,h} + h\| \utilde{\bs{\sigma}}^j \|_1, \\
\label{eq:thetap-estm2} \| \pd_t \theta_p^j \|_0 &\lesssim \| \pd_t \theta_{\bs{u}}^j \|_{1,h} + \| \pd_t e_{\bs{u}}^{I, j} \|_{1,h} + h\| \pd_t \utilde{\bs{\sigma}}^j \|_1. 
\end{align}
\end{lemma}
\begin{proof}
By the inf-sup condition \eqref{eq:inf-sup1} there exists a $\bs{v} \in V_h$ such that 
\begin{align*}
(\theta_p^j, \div \bs{v}) = \| \theta_p^j \|_0^2, \qquad \| \bs{v} \|_{1,h} \lesssim \| \theta_p^j \|_0.
\end{align*}
Using this $\bs{v}$ in \eqref{eq:theta-j-eq1}, we have 
\begin{align*}
\| \theta_p^j \|_0^2 &= a_h( \theta_{\bs{u}}^j +  e_{\bs{u}}^{I,j}, \bs{v}) + E_h(  \utilde{\bs{\sigma}}^j , \bs{v} ) \\
&\lesssim (\|  \theta_{\bs{u}}^j \|_{1,h} + \| e_{\bs{u}}^{I,j} \|_{1,h} + h \|  \utilde{\bs{\sigma}}^j \|_1 ) \| \bs{v} \|_{1,h} \\
&\lesssim (\|  \theta_{\bs{u}}^j \|_{1,h} + \| e_{\bs{u}}^{I,j} \|_{1,h} + h \|  \utilde{\bs{\sigma}}^j \|_1 ) \|  \theta_p^j \|_0, 
\end{align*}
so the inequality \eqref{eq:thetap-estm1} follows. 

The inequality \eqref{eq:thetap-estm2} follows by taking the difference of \eqref{eq:theta-j-eq1} at $t = t_{j}$ and at $t = t_{j-1}$, and applying a similar argument. 
\end{proof}

Now Theorem \ref{thm:full-error-estimate} is an immediate consequence of the following and the triangle inequality.
\begin{thm} \label{lemma:full-theta-thm} Suppose that $(\bs{u}, \bs{z},p)$ is an exact solution of {\rm{(\ref{eq:new-strong-eq1}--\ref{eq:new-strong-eq3})}} with sufficient regularity and $\theta_{\bs{u}}^j$, $\theta_{\bs{z}}^j$, $\theta_p^j$ are defined as in {\rm{(\ref{eq:err-decomp1}--\ref{eq:err-decomp3})}} for a fully discrete solution $\{(\bs{U}^j, \bs{Z}^j, P^j)\}_{1\leq j \leq N}$ with initial data satisfying \eqref{eq:initial-data-approx}. Suppose also that $M_1$ and $M_2$ are defined as in Theorem \ref{thm:full-error-estimate}. Then 
\begin{align}
\label{eq:full-theta-estm1} \max_{1 \leq i \leq N} \| \theta_{\bs{u}}^i \|_{1,h} + \max_{1 \leq i \leq N} \| \theta_p^{i} \|_0 &\leq c(\lap t + h) M_1, \\
\notag \max_{1 \leq i \leq N} \| \theta_{\bs{z}}^i |_{0} &\leq c(\lap t + h) M_2,
\end{align}
with constants $c$ independent of $T_0$. 
\end{thm}
\begin{proof}
Taking $\bs{v} = \theta_{\bs{u}}^{j+1} - \theta_{\bs{u}}^j$, $\bs{w} = \lap t \theta_{\bs{z}}^{j+1}$, $q = \lap t \theta_p^{j+1}$ in (\ref{eq:theta-j-eq1}--\ref{eq:theta-j-eq3}), and adding these equations together yield
\begin{align} \label{eq:theta-j-eq4}
& \| \theta_{\bs{u}}^{j+1} \|_{a,h}^2 + \| \theta_p^{j+1} \|_{c_0}^2 + \lap t \| \theta_{\bs{z}}^{j+1} \|_0^2 - a_h(\theta_{\bs{u}}^{j+1}, \theta_{\bs{u}}^{j}) - (c_0 \theta_p^{j+1}, \theta_p^j) \\
\notag &= - a_h(e_{\bs{u}}^{I,j+1}, \theta_{\bs{u}}^{j+1} - \theta_{\bs{u}}^j) - E_h(\utilde{\bs{\sigma}}^{j+1}, \theta_{\bs{u}}^{j+1} - \theta_{\bs{u}}^j) - \lap t (e_{\bs{z}}^{I,j+1}, \theta_{\bs{z}}^{j+1})  \\
\notag &\quad + \lap t (\omega_1^{j+1} + \omega_2^{j+1}, \theta_p^{j+1}) \\
\notag &=: \Phi_1^{j+1} + \Phi_2^{j+1} + \Phi_3^{j+1} + \Phi_4^{j+1}. 
\end{align}
By the Cauchy--Schwarz and the arithmetic-geometric mean inequalities, 
\begin{align*}
a_h(\theta_{\bs{u}}^{j+1}, \theta_{\bs{u}}^{j}) + (c_0 \theta_p^{j+1}, \theta_p^j) \leq \half ( \| \theta_{\bs{u}}^{j+1} \|_{a,h}^2 + \| \theta_{\bs{u}}^j \|_{a.h}^2 + \| \theta_p^{j+1} \|_{c_0}^2 + \| \theta_p^j \|_{c_0}^2). 
\end{align*}
Applying this to \eqref{eq:theta-j-eq4}, after some algebraic manipulations, we have 
\begin{multline*}
\half (\| \theta_{\bs{u}}^{j+1} \|_{a,h}^2 + \| \theta_p^{j+1} \|_{c_0}^2 ) + \lap t \| \theta_{\bs{z}}^{j+1} \|_0^2 \\
\leq \half ( \| \theta_{\bs{u}}^j \|_{a,h}^2 + \| \theta_p^j \|_{c_0}^2) + \Phi_1^{j+1} + \Phi_2^{j+1} + \Phi_3^{j+1} + \Phi_4^{j+1}. 
\end{multline*}
As in the semidiscrete error analysis, the proof consists of three parts dealing with $\theta_{\bs{u}}^j$, $\theta_{\bs{z}}^j$, $\theta_{p}^j$, respectively.

\noindent {\bf Estimate of $\max_{1\leq i \leq N} \| \theta_{\bs{u}}^i \|_{1,h}$} : Defining 
\begin{align}
X_i^2 := \half( \| \theta_{\bs{u}}^i \|_{a,h}^2 + \| \theta_p^i \|_{c_0}^2), \qquad Y_i^2 := \lap t \sum_{j=1}^{i} \| \theta_{\bs{z}}^{j} \|_0^2,
\end{align}
and by taking the summation of \eqref{eq:theta-j-eq4} over $1 \leq j \leq i $, we have 
\begin{align*}
X_{i}^2 + Y_{i}^2 =  X_0^2 + \sum_{j=1}^{i} \left(  \Phi_1^{j} + \Phi_2^{j} + \Phi_3^{j} + \Phi_4^{j} \right),
\end{align*}
for all $1 \leq i \leq N$. Let us define $\bar i$ to be $X_{\bar i} = \max_{1 \leq j \leq N} X_j$. By comparing $\sum_{j=1}^{\bar i} (  \Phi_1^{j} + \Phi_2^{j} + \Phi_4^{j} )$ and $\sum_{j=1}^{\bar i} \Phi_3^{j}$,
we have
\begin{align}
\label{eq:full-estm-case1} &\text{either}& X_{\bar i}^2 + Y_{\bar i}^2 &\leq X_0^2 + 2\sum_{j=1}^{\bar i} \Phi_3^{j}, \\
\label{eq:full-estm-case2} &\text{or} &X_{\bar i}^2 + Y_{\bar i}^2 &\leq  X_0^2 + 2\sum_{j=1}^{\bar i} \left(  \Phi_1^{j} + \Phi_2^{j} + \Phi_4^{j} \right). 
\end{align}
{\it Case I} : Suppose that \eqref{eq:full-estm-case1} holds. The Cauchy--Schwarz and the arithmetic-geometric mean inequalities give
\begin{align*}
\sum_{j=1}^{\bar i} \Phi_3^{j} \leq \frac{\lap t}{2} \sum_{j=1}^{\bar i} \| e_{\bs{z}}^{I,j} \|_0^2 + \frac{\lap t}{2} \sum_{j=1}^{\bar i} \| \theta_{\bs{z}}^{j} \|_0^2 = \frac{\lap t}{2} \sum_{j=1}^{\bar i} \| e_{\bs{z}}^{I,j} \|_0^2  + \half Y_{\bar i}^2.
\end{align*}
Applying this to \eqref{eq:full-estm-case1} yields
\begin{align*} 
X_{\bar i}^2 &\leq X_0^2 + {\lap t} \sum_{j=0}^{\bar i - 1} \| {e}_{\bs{z}}^{I,j+1} \|_0^2 \\
\notag &\lesssim X_0^2 + \int_0^{\bar i \lap t} \| \dot e_{\bs{z}}^I (s)\|^2 ds \\
\notag &\lesssim h^2(\| \bs{u}(0) \|_2^2 + \| p(0) \|_1^2 + \| p \|_{W^{1,2} H^2}^2). 
\end{align*}
{\it Case II} : Suppose that \eqref{eq:full-estm-case2} holds. We remark a summation by parts identity
\begin{align*} 
- \sum_{j=1}^{\bar i} F^{j}(G^{j} - G^{j-1}) = - F^{\bar i} G^{\bar i} + F^0 G^0 +  \sum_{j=1}^{\bar i} (F^{j} - F^{j-1}) G^{j-1}  . 
\end{align*}
Using this, we get 
\begin{align} \label{eq:Phi1-sum-modified}
\sum_{j=1}^{\bar i} \Phi_1^{j} &= -  \sum_{j=1}^{\bar i} a_h( e_{\bs{u}}^{I,j},\theta_{\bs{u}}^{j} - \theta_{\bs{u}}^{j-1}) \\
\notag &= - a_h(e_{\bs{u}}^{I,\bar i}, \theta_{\bs{u}}^{\bar i}) + a_h(e_{\bs{u}}^{I, 0}, \theta_{\bs{u}}^0))
+ \sum_{j=1}^{\bar i} a_h(e_{\bs{u}}^{I,j} - e_{\bs{u}}^{I,j-1}, \theta_{\bs{u}}^{j-1}) . 
\end{align}
Considering the identity
\begin{align*}
e_{\bs{u}}^{I,j} - e_{\bs{u}}^{I,j-1} = \int_{t_{j-1}}^{t_j} \dot e_{\bs{u}}^I (s) ds, 
\end{align*}
and applying the Cauchy--Schwarz inequality to the last form in \eqref{eq:Phi1-sum-modified}, 
\begin{align*}
\sum_{j=1}^{\bar i} \Phi_1^{j} &\lesssim \| e_{\bs{u}}^{I, \bar i} \|_{1,h} \| \theta_{\bs{u}}^{\bar i} \|_{1,h} + \| e_{\bs{u}}^{I, 0} \|_{1,h} \| \theta_{\bs{u}}^{0} \|_{1,h} \\
& \quad + \sum_{j=1}^{\bar i} \left( \int_{t_{j-1}}^{t_j} \| \dot e_{\bs{u}}^I (s) \|_{1,h} ds \right) \| \theta_{\bs{u}}^{j-1} \|_{1,h}.
\end{align*}
By the definition of $X_{\bar i}$ and the interpolation error estimates we have 
\begin{align} \label{eq:Phi1-sum-modified2}
\sum_{j=1}^{\bar i} \Phi_1^{j} \leq C_0 X_{\bar i}, 
\end{align}
with $C_0 = ch( \| \bs{u} (t_{\bar i}) \|_2 + \| \bs{u} (0) \|_2 + \| \bs{u} \|_{W^{1,1}([0, t_{\bar i}]; H^2)})$. Similarly, we can have 
\begin{align}  \label{eq:Phi2-sum-modified}
\sum_{j=1}^{\bar i} \Phi_2^{j} \leq C_1 X_{\bar i},
\end{align}
with $C_1 = ch(\| \utilde{\bs{\sigma}}(0) \|_1 + \| \utilde{\bs{\sigma}}(t_{\bar i}) \|_1 + \| \utilde{\bs{\sigma}} \|_{W^{1,1}([0, t_{\bar i}]; H^1)})$. 
By the definition of $\Phi_4^{j}$ in \eqref{eq:theta-j-eq4}, and \eqref{eq:thetap-estm1}, one can see that 
\begin{align} \label{eq:Phi4-modified}
\Phi_4^{j} &\leq \lap t \| \omega_1^{j} + \omega_2^{j} \|_{0} ( \| \theta_{\bs{u}}^j \|_{1,h} +  \| e_{\bs{u}}^{I,j}\|_{1,h} + h\| \utilde{\bs{\sigma}}^{j} \|_1) \\
\notag &\leq c \lap t \| \omega_1^{j} + \omega_2^{j} \|_0 X_{\bar i} + h \lap t \| \omega_1^{j} + \omega_2^{j} \|_0(\| \bs{u} \|_{L^\infty H^2} + \| \utilde{\bs{\sigma}} \|_{L^\infty H^1}). 
\end{align}
for $1 \leq j \leq \bar i$. Using \eqref{eq:Phi1-sum-modified2}, \eqref{eq:Phi2-sum-modified}, \eqref{eq:Phi4-modified} to \eqref{eq:full-estm-case2} yields 
\begin{align*}
X_{\bar i}^2 + Y_{\bar i}^2 &\leq \left( 2C_0 + 2C_1 + c \lap t \sum_{j=1}^{\bar i} \| \omega_1^{j} + \omega_2^{j} \|_0 \right) X_{\bar i} \\
\notag &\quad + X_0^2 + 2 h \lap t \sum_{j=1}^{\bar i} \| \omega_1^j + \omega_2^j \|_0 (\| \bs{u} \|_{L^\infty H^2} + \| \utilde{\bs{\sigma}} \|_{L^\infty H^1}) \\
\notag &=: C X_{\bar i} + D.
\end{align*}
By Lemma \ref{lemma:basic-ineq},
$X_{\bar i} + Y_{\bar i} \lesssim \max \{4 C, 2 \sqrt{D} \}$, and
we need to show that 
\begin{align*}
C \lesssim \lap t + h, \qquad D \lesssim (\lap t + h)^2.
\end{align*} 
Since we already know that $C_0 + C_1 + X_0 \lesssim h(\| \bs{u} \|_{W^{1,1} H^2} + \| p \|_{W^{1,1}H^1})$, these inequalities are proven if we show  
\begin{multline} \label{eq:omega23-estm}
\lap t \sum_{j=1}^{\bar i} \| \omega_1^j + \omega_2^j \|_0 \\
\lesssim (h + \lap t)( \| \bs{u} \|_{W^{2,1}H^1} + \| \bs{u} \|_{W^{1,1}H^2} + \| p \|_{W^{2,1}L^2} + \| p \|_{W^{1,1}H^1}),
\end{multline}
which then completes the proof of \eqref{eq:full-theta-estm1} for $\| \theta_{\bs{u}}^i \|_{1,h}$. 

To show \eqref{eq:omega23-estm}, by the definition of $\omega_1^j$ and Taylor expansion, we first have  
\begin{align*}
\lap t \| \omega_1^j \|_0 &\lesssim \lap t (\| \pd_t p^j - \dot p^j \|_0 + \| \pd_t e_p^{I,j} \|_0) \\
&\leq  \left \| \int_{t_{j-1}}^{t_j} (\dot p(s) - (\lap t) \dot p(t_{j})) ds \right \|_0 + \left \| \int_{t_{j-1}}^{t_j} \dot e_p^I(s)  ds \right \|_0 \\
&\leq \lap t \int_{t_{j-1}}^{t_j} \| \ddot p (s) \|_0 ds + h \int_{t_{j-1}}^{t_j} \| \dot p (s) \|_1 ds.
\end{align*}
A similar argument yields 
\begin{align*}
\lap t \| \omega_2^j \|_0 \leq \lap t \int_{t_{j-1}}^{t_j} \| \ddot{\bs{u}}(s) \|_1 ds + h \int_{t_{j-1}}^{t_j} \| \dot{\bs{u}} (s) \|_2 ds.
\end{align*}
Then \eqref{eq:omega23-estm} follows by the triangle inequality and taking the summation of the above inequalities for $\lap t \| \omega_1^j \|_0$ and $\lap t \| \omega_2^j \|_0$ over $1 \leq j \leq \bar i$.

\noindent {\bf Estimate of $\max_{1\leq i \leq N} \| \theta_{p}^i \|_0$} : By \eqref{eq:thetap-estm1} and the estimate of $\| \theta_{\bs{u}}^j \|_{1,h}$, 
\begin{align*}
\max_{1\leq i\leq N} \| \theta_p^i \|_0 &\lesssim \max_{1\leq i\leq N} (\| \theta_{\bs{u}}^i \|_{1,h} + \| e_{\bs{u}}^{I, i} \|_{1,h} + h\| \utilde{\bs{\sigma}}^i \|_1) \\
&\lesssim \max_{1\leq i\leq N} \| \theta_{\bs{u}}^i \|_{1,h} + h \| \bs{u}, \utilde{\bs{\sigma}} \|_{L^\infty H^1}.
\end{align*}
The conclusion follows from the estimate of $\| \theta_{\bs{u}}^i \|_{1,h}$. 

\noindent {\bf Estimate of $\max_{1\leq i \leq N} \| \theta_{\bs{z}}^i \|_0$} : 
The time steppings of \eqref{eq:theta-j-eq1} and \eqref{eq:theta-j-eq2} as in \eqref{time-quotient} give
\begin{align*} 
a_h(\pd_t {\theta}_{\bs{u}}^{j+1}, \bs{v}) - (\pd_t {\theta}_p^{j+1}, \div \bs{v}) &= -a_h(\pd_t {e}_{\bs{u}}^{I,j+1}, \bs{v}) - E_h( \pd_t {\utilde{\bs{\sigma}}}^{j+1} , \bs{v} ),  \\
( \pd_t \theta_{\bs{z}}^{j+1}, \bs{w}) + ( \pd_t \theta_p^{j+1}, \div \bs{w}) &= - ( \pd_t e_{\bs{z}}^{I, j+1}, \bs{w}).
\end{align*}
Taking $\bs{v} = \pd_t \theta_{\bs{u}}^{j+1}$, $w = \theta_{\bs{z}}^{j+1}$ in the above, $q = \pd_t \theta_p^{j+1}$ in \eqref{eq:theta-j-eq3}, and adding these three equations together, we have
\begin{multline*}
\| \pd_t \theta_{\bs{u}}^{j+1} \|_{a,h}^2 + (\pd_t \theta_{\bs{z}}^{j+1}, \theta_{\bs{z}}^{j+1}) + \| \pd_t \theta_p^{j+1} \|_{c_0}^2 \\
= -a_h( \pd_t e_{\bs{u}}^{I,j+1}, \pd_t \theta_{\bs{u}}^{j+1}) - E_h(\pd_t \utilde{\bs{\sigma}}^{j+1}, \pd_t \theta_{\bs{u}}^{j+1}) \\
- (\pd_t e_{\bs{z}}^{I,j+1}, \theta_{\bs{z}}^{j+1}) + (\omega_1^{j+1} + \omega_2^{j+1}, \pd_t \theta_p^{j+1}). 
\end{multline*}
Observe that $\| \pd_t \theta_p^{j+1} \|_0$ can be estimated by $\| \pd_t \theta_{\bs{u}}^{j+1} \|_{a,h} $ with some extra terms by \eqref{eq:thetap-estm2}. Applying the Cauchy--Schwarz inequality to the above, and then applying Young's inequality regarding \eqref{eq:thetap-estm2}, we have 
\begin{multline*}
(\pd_t \theta_{\bs{z}}^{j+1}, \theta_{\bs{z}}^{j+1}) \\
\lesssim \| \pd_t e_{\bs{u}}^{I,j+1} \|_{1,h}^2 + h^2 \| \pd_t \utilde{\bs{\sigma}}^{j+1} \|_1^2
+ ( \pd_t e_{\bs{z}}^{I,j+1}, \theta_{\bs{z}}^{j+1}) + \| \omega_1^{j+1} + \omega_2^{j+1} \|_0^2 .
\end{multline*}
Suppose that $\max_{1\leq i \leq N} \theta_{\bs{z}}^i = \theta_{\bs{z}}^{\bar i}$. Taking the summation of the above inequality from $j=0$ to $j=\bar i-1$,
\begin{multline} \label{eq:theta-interm-ineq1}
\sum_{j=1}^{\bar i} (\pd_t \theta_{\bs{z}}^{j}, \theta_{\bs{z}}^{j}) \\
\lesssim \sum_{j=1}^{\bar i} \left( \| \pd_t e_{\bs{u}}^{I,j} \|_{1,h}^2 + h \| \pd_t \utilde{\bs{\sigma}}^{j} \|_1^2 + ( \pd_t e_{\bs{z}}^{I,j}, \theta_{\bs{z}}^{j}) + \| \omega_1^{j} + \omega_2^{j} \|_0^2  \right) .
\end{multline}
By the arithmetic-geometric mean inequality, 
\begin{align*}
\lap t \sum_{j=1}^{\bar i} (\pd_t \theta_{\bs{z}}^{j}, \theta_{\bs{z}}^{j}) &= \half \| \theta_{\bs{z}}^{\bar i} \|_0^2 - \half \| \theta_{\bs{z}}^0 \|_0^2 + \half \sum_{j=1}^{\bar i} \{ \| \theta_{\bs{z}}^{j-1} \|_0^2 + \| \theta_{\bs{z}}^j \|_0^2 - (\theta_{\bs{z}}^{j-1}, \theta_{\bs{z}}^j ) \}\\
&\geq \half \| \theta_{\bs{z}}^{\bar i} \|_0^2 - \half \| \theta_{\bs{z}}^0 \|_0^2.
\end{align*}
Together with the above inequality \eqref{eq:theta-interm-ineq1} yields
\begin{align*}
&\half \| \theta_{\bs{z}}^{\bar i} \|_0^2 - \half \| \theta_{\bs{z}}^0 \|_0^2 \\
&\lesssim \lap t \sum_{j=1}^{\bar i} \{ \| \pd_t e_{\bs{u}}^{I,j} \|_{1,h}^2 + h^2 \| \pd_t \utilde{\bs{\sigma}}^{j} \|_1^2
+ ( \pd_t e_{\bs{z}}^{I,j}, \theta_{\bs{z}}^{j}) + \| \omega_1^{j} + \omega_2^{j} \|_0^2  \} \\
&\lesssim \lap t \sum_{j=1}^{\bar i} \{ \| \pd_t e_{\bs{u}}^{I,j} \|_{1,h}^2 + h^2 \| \pd_t \utilde{\bs{\sigma}}^{j} \|_1^2
+ \| \pd_t e_{\bs{z}}^{I,j} \|_0 \| \theta_{\bs{z}}^{\bar i} \|_0 + \| \omega_1^{j} + \omega_2^{j} \|_0^2  \},
\end{align*}
where the last inequality is due to the Cauchy--Schwarz inequality and the definition of $\| \theta_{\bs{z}}^{\bar i} \|_0$. This inequality is in a form that Lemma \ref{lemma:basic-ineq} is applicable to 
\begin{align*}
&A = \frac{1}{\sqrt{2}} \| \theta_{\bs{z}}^{\bar i} \|_0, \qquad B=0, \qquad C = c \lap t \sum_{j=1}^{\bar i} \| \pd_t e_{\bs{z}}^{I,j} \|_0, \\
&D = \half \| \theta_{\bs{z}}^0 \|_0^2 + c \lap t \sum_{j=1}^{\bar i} \{ \| \pd_t e_{\bs{u}}^{I,j} \|_{1,h}^2 + h^2 \| \pd_t \utilde{\bs{\sigma}}^{j} \|_1^2 + \| \omega_1^{j} + \omega_2^{j} \|_0^2  \}. 
\end{align*}
By Lemma \ref{lemma:basic-ineq}, we only need to show that $C \lesssim h + \lap t$ and $D \lesssim h^2 + \lap t^2$. 

The estimate of $C$ is easily obtained by
\begin{align*}
C = c \sum_{j=1}^{\bar i} \left \| \int_{t_{j-1}}^{t_j} \dot e_{\bs{z}}^{I}(s) ds \right \|_0 \lesssim h \| \bs{z} \|_{W^{1,1}H^1} \lesssim h \| p \|_{W^{1,1}H^2}.
\end{align*}
By the assumption on initial data there is nothing to prove for the first term of $D$. For estimates of the other terms of $D$ we will prove the following inequalities:
\begin{align}
\label{eq:D-estm1} \lap t \sum_{j=1}^{\bar i} \| \pd_t e_{\bs{u}}^{I,j} \|_{1,h}^2 
&\lesssim h^2 \| \bs{u} \|_{W^{1,2}H^2}^2, \\
\label{eq:D-estm2} \lap t \sum_{j=1}^{\bar i} \| \pd_t \utilde{\bs{\sigma}}^{j} \|_1^2 &\lesssim h^2 \| \utilde{\bs{\sigma}} \|_{W^{1,2}H^1}^2 \lesssim h^2 (\| \bs{u} \|_{W^{1,2}H^2}^2 + \| p \|_{W^{1,2}H^1}^2), \\
\label{eq:D-estm3} \lap t \sum_{j=1}^{\bar i} \| \omega_1^j \|_0^2 &\lesssim \lap t^2 \| p \|_{W^{2,2}L^2}^2 + h^2 \| p \|_{W^{1,2}H^1}^2, \\
\label{eq:D-estm4} \lap t \sum_{j=1}^{\bar i} \| \omega_2^j \|_0^2 &\lesssim \lap t^2 \| \bs{u} \|_{W^{2,2}H^1}^2 + h^2 \| \bs{u} \|_{W^{1,2}H^2}^2 .
\end{align} 
The proofs of \eqref{eq:D-estm1} and \eqref{eq:D-estm2} are similar, so we only show \eqref{eq:D-estm1}. By H\"{o}lder inequality and \eqref{eq:interp-approx} we have
\begin{align*}
\lap t \| \pd_t e_{\bs{u}}^{I,j} \|_{1,h}^2 = \frac{1}{\lap t} \left \| \int_{t_{j-1}}^{t_j} \dot e_{\bs{u}}^{I}(s) ds \right \|_{1,h}^2 &\leq \int_{t_{j-1}}^{t_j} \| \dot e_{\bs{u}}^I(s) \|_{1,h}^2 ds \\
&\lesssim h^2 \int_{t_{j-1}}^{t_j} \| \dot{\bs{u}} (s)\|_2^2 ds,
\end{align*}
and its summation over $1 \leq j \leq \bar i$ gives \eqref{eq:D-estm1}.

Since proofs of \eqref{eq:D-estm3} and \eqref{eq:D-estm4} are similar as well, we only prove \eqref{eq:D-estm3}. For each $1 \leq j \leq \bar i$, 
\begin{align*}
\lap t \| \omega_1^j \|_0^2 &\lesssim \lap t \left \| \frac{1}{\lap t} \int_{t_{j-1}}^{t_j} (\dot p(s) - (\lap t) \dot p(t_j)) ds - \frac{1}{\lap t} \int_{t_{j-1}}^{t_j} \dot e_p^I (s) ds \right \|_0^2 \\ 
&\lesssim \frac{1}{\lap t} \left \| \int_{t_{j-1}}^{t_j} (\dot p(s) - (\lap t) \dot p(t_j)) ds \right \|_0^2 + \frac{1}{\lap t} \left \| \int_{t_{j-1}}^{t_j}\dot e_p^I (s) ds \right \|_0^2 \\
&\lesssim \lap t \left( \int_{t_{j-1}}^{t_j} \| \ddot{p}(s) \|_0 ds \right)^2 + \frac{h^2}{\lap t} \left ( \int_{t_{j-1}}^{t_j} \| \dot p (s) \|_0 ds \right )^2.
\end{align*}
Applying H\"{o}lder's inequality to the last form yields
\begin{align*}
\lap t \| \omega_1^j \|_0^2 &\lesssim \lap t^2 \int_{t_{j-1}}^{t_j} \| \ddot{p}(s) \|_0^2 ds + h^2 \int_{t_{j-1}}^{t_j} \| \dot p (s) \|_1^2 ds ,
\end{align*}
and \eqref{eq:D-estm3} follows by taking the summation of it over $1 \leq j \leq \bar i$. 
\end{proof}

\section{Conclusion} \label{sec:conclusion}
We proposed a new finite element method for Biot's consolidation model and showed the a priori error analysis of the semidiscrete and the fully discrete solutions. Uniform-in-time error bounds of all the unknowns are obtained with the method, so the poroelasticity locking does not occur. 
Moreover, we do not use Gr\"{o}nwall's inequality in our error analysis, so there is no exponentially growing factor in the obtained error bounds. 

We remark that our method can be readily extended to higher order and to rectangular meshes using finite elements having the features (\ref{eq:sigma-property1}--\ref{eq:sigma-property3}). For instance, Guzm\'{a}n and Neilan constructed higher order elements in two and three dimensions for triangular meshes in \cite{MR2991835}. 
Chen, Xie, and Zhang constructed some two and three dimensional rectangular elements in \cite{MR2495068}. 

\bibliography{./../reference/FEM}
\bibliographystyle{amsplain}
\vspace{.125in}

\end{document}